\documentclass{article}
\usepackage{amsmath,amssymb,amsthm}
\usepackage{calligra,indentfirst,epsfig}
\calligra \newcommand{\z}{\mathbb{Z}_{2^m}}
\newcommand{\co}{\mathcal{C}}
\newcommand{\f}{\frac{m}{2}}
\newcommand{\fo}{\frac{m-1}{2}}
\newcommand{\fe}{\frac{m+1}{2}}
\newtheorem{theo}{Theorem}[section]
\newtheorem{lem}{Lemma}[section]

\theoremstyle{remark}
\newtheorem{rem}{Remark}[section]
\theoremstyle{definition}
\newtheorem{defin}{Definition}[section]
\newtheorem{prop}{Proposition}[section]
\begin{document}
\title{Construction of self-dual codes over $\mathbb{Z}_{2^m}$}
\author{Anuradha Sharma{\footnote{Corresponding Author, Email address: anuradha@maths.iitd.ac.in}, Amit K. Sharma}\\
{\it Department of Mathematics}\\{\it Indian Institute of
Technology Delhi}\\{\it New Delhi 110016, India}}
\date{}
\maketitle
\begin{abstract}
Self-dual codes (Type I and Type II codes) play an important role in the construction of even unimodular lattices, and hence in the determination of Jacobi forms. In this paper, we construct both Type I and Type II codes (of higher lengths) over the ring $\z$ of integers modulo $2^m$ from shadows of Type I codes of length $n$ over $\z$ for each positive integer $n,$ and obtain their complete weight enumerators. Using these results, we also determine some Jacobi forms on the modular group $\Gamma(1)=SL(2,\mathbb{Z}).$ Besides this, for each positive integer $n,$ we also construct self-dual codes (of higher lengths)  over $\z$ from the generalized shadow of a self-dual code $\co$ of length $n$ over $\z$ with respect to a vector $s \in \z^n \setminus \co $ satisfying either $s\cdot s \equiv 0~(\text{mod }2^m)$ or $s \cdot s \equiv 2^{m-1}~(\text{mod }2^m).$
\end{abstract}
\section{Introduction}
Self-dual codes over finite rings and their shadows have been an interesting object of study for a long time due to their connection with the theory of unimodular lattices and modular forms such as Jacobi forms, elliptic modular forms, Siegel modular forms. Moreover, shadows and generalized shadows of self-dual codes are useful in the construction of self-dual codes of higher lengths.

Conway and Sloane \cite{conslon} first defined the shadow of a binary self-dual code and used it to obtain an upper bound on the minimal distance of binary self-dual codes. From the shadow of a binary Type I code of length $n,$ Brualdi and Pless \cite{bru} constructed a binary self-dual code $\co_1$ of length $n+2$ when $n\equiv 2\text{ or }6~(\text{mod }8)$ and  a binary self-dual code $\co_2$ of length $n+4$ when $n\equiv 0\text{ or }4~(\text{mod }8),$ and obtained weight enumerators of the codes $\co_1$ and $\co_2.$  They also observed that the code $\co_1$ is Type I when $n \equiv 2~(\text{mod }8)$ and is Type II when $n \equiv 6~(\text{mod }8),$ whereas the code $\co_2$ is Type I when $n \equiv 0~(\text{mod }8)$ and is Type II when $n \equiv 4~(\text{mod }8).$ In the same work, they also studied the shadow of a binary Type II code $\co$ of length $n$ by considering a subcode $\co_0$ of $\co$ with codimension 1. From this, they obtained a binary self-dual code of length $n+4$ provided $\textbf{1} \in \co_0$ and a binary self-dual code of length $n+2$ provided $\textbf{1} \not\in \co_0,$ and computed weight enumerators of both the codes. Later, Tsai \cite{han} defined (generalized) shadow of a binary self-dual code $\co$ of length $n$ with respect to a vector $s \in \mathbb{Z}_2^n \setminus \co.$ Using this, he constructed a binary self-dual code of length $n+2$ when $s \cdot s \equiv 1~(\text{mod }2)$ and a binary self-dual code of length $n+4$ when $s \cdot s \equiv 0~(\text{mod }2),$ and computed their weight enumerators. In an attempt to generalize the construction method proposed by Brualdi and Pless \cite{bru} and Tsai \cite{han}, Dougherty et al.  \cite{shadowz4} studied the shadow of a Type I code of length $n$ over the ring $\mathbb{Z}_4$ of integers modulo $4.$ From this, they constructed Type I codes (of higher lengths) over $\mathbb{Z}_4$ when $n\equiv 0,1,2,3~(\text{mod }8)$ and Type II codes (of higher lengths) over $\mathbb{Z}_4$ when $n\equiv 4,5,6,7~(\text{mod }8).$ In the same work, they introduced the notion of generalized shadow of a self-dual code $\co$ of length $n$ over $\mathbb{Z}_4$ with respect to a vector $s \in \mathbb{Z}_4^n \setminus \co$ whose components are either 0 or 2, and used it to construct a self-dual code of length $n+4$ over $\mathbb{Z}_4.$ While doing so, they constructed a code $\co^*$ over $\mathbb{Z}_4$ and claimed that $\co^*$ is a self-orthogonal code. However, the authors observed that the code $\co^*$ is not a linear code over $\mathbb{Z}_4,$ which led to errors in Theorems 3.14-3.16. The aim of this paper is to rectify these errors and to generalize the construction method for codes over the ring $\z$ ($m \geq 1$ is an integer) of integers modulo $2^m.$

On the other hand, Bannai et al. \cite{bannai} studied self-dual codes over the ring $\mathbb{Z}_{2m}$ ($m \geq 1$ is an integer) of integers modulo $2m$ and extended the notion of shadow of a Type I code over $\mathbb{Z}_4$ to Type I codes over  $\mathbb{Z}_{2m}.$ In the same work, they also constructed Siegel modular forms from complete and symmetrized weight enumerators in genus $g$ of Type II codes over $\mathbb{Z}_{2m}.$ Choie and Kim \cite{choie} constructed Jacobi forms from  complete weight enumerators of Type II codes over $\mathbb{Z}_{2m}.$

In this paper, we generalize the construction method employed by Brualdi and Pless \cite{bru} and Tsai \cite{han}, and rectify an error in the construction method proposed by Dougherty et al. \cite{shadowz4}. Here we  obtain both Type I and Type II codes (of higher lengths) over $\z$ from shadows of Type I codes of length $n$ over $\z$ for all $n.$ Also for each positive integer $n,$ we construct  self-dual codes (of higher lengths) over $\z$ from the generalized shadow of a self-dual code $\co$ of length $n$  over $\z$ with respect to a vector $s \in \z^n \setminus \co$ satisfying either $s\cdot s\equiv 0 ~(\text{mod }2^m)$ or $s\cdot s\equiv 2^{m-1}~(\text{mod }2^m).$ We also determine complete weight enumerators of the codes constructed above.  As an application of these results, we also determine some Jacobi forms on the modular group $\Gamma(1)=SL(2,\mathbb{Z}).$

This paper is organized as follows: In Section \ref{secprelim}, we recall some basic definitions, discuss the shadow of a Type I code over $\z$ and the generalized shadow of a self-dual code $\co$ over $\z$ with respect to a vector $s\in\z^n\setminus\co$ satisfying either $s\cdot s\equiv 0~(\text{mod }2^m)$ or $s\cdot s\equiv 2^{m-1}~(\text{mod }2^m).$ In Section \ref{secconsshad}, we first construct a self-orthogonal code (of higher length) over $\z$ from a self-dual code of length $n$ over $\z$ (Proposition \ref{p}). Using this, we construct Type I and Type II codes (of higher lengths) over $\z$ from the shadow of a Type I code of length $n$ over $\z$ for each positive integer $n,$ and compute their complete weight enumerators (Theorems \ref{theosneven} and \ref{theosnodd}).  We also determine some Jacobi forms on the modular group $\Gamma(1)$ from complete weight enumerators  of Type II codes constructed in the respective cases (Theorems \ref{theosnej} and \ref{theosnoj}). In Section \ref{secgenshad}, we construct self-dual codes (of higher lengths) over $\z$ using the generalized shadow  of a self-dual code $\co$ of length $n$ over $\z$ with respect to a vector $s\in \z^n\setminus\co $ for each positive integer $n,$ provided $s\cdot s\equiv 0 \text{ or }2^{m-1}~(\text{mod }2^m)$ (Theorems \ref{theogs0} and \ref{theogsn0}). Here also, we compute complete weight enumerators of self-dual codes constructed in the respective cases.
\section{Some preliminaries}\label{secprelim}
For positive integers $m$ and $n,$ let $\mathbb{Z}_{2^m}^n$ denote the $\z$-module consisting of  all $n$-tuples over the ring $\mathbb{Z}_{2^m}=\{0,1,2,\cdots,2^m-1\}$ of integers modulo $2^m.$ Then a linear code $\mathcal{C}$ of length $n$ over $\z$ is defined as an additive subgroup of $\z^n$ and its elements are called codewords. The size of $\co$ is defined as the total number of codewords in $\co$ and is denoted by $|\co|.$
Furthermore, the dual code of $\co$ is defined as the set $\co^{\perp}=\{v\in\z^n:u\cdot v=0\text{ for all }u\in \mathcal{C}\},$ where $u\cdot v$ denotes the standard bilinear form in $\z^n.$ Observe that the dual code $\co^{\perp}$ is also a linear code of length $n$ over $\z.$ A linear code $\co$ is said to be self-orthogonal if $\co\subseteq\co^{\perp},$ whereas the code $\co$ is said to be self-dual if $\co=\co^{\perp}.$ Then the following result is well-known.
\begin{lem}\label{a} If $\co$ is a self-dual code of length $n$ over $\z,$ then $mn$ is an even integer. \end{lem}
\begin{proof} For proof, see Theorem 2.3 of Dougherty et al. \cite{dough2pm}. \end{proof}
Next to define Type I and Type II codes over $\z,$ we need to define the Euclidean weight in $\z.$ The Euclidean weight of an element $a\in\z,$ denoted by $wt_E(a),$ is defined as $wt_E(a)=\min\{a^2,(2^m-a)^2\}.$ Furthermore, the Euclidean weight of a vector $v=(v_1,v_2,\cdots,v_n)\in\z^n$ is defined as $wt_E(v)=\sum\limits_{j=1}^{n}wt_E(v_j).$
Then a self-dual code $\co$ is said to be a Type II code if the Euclidean weight of each of its codewords is divisible by $2^{m+1},$ otherwise $\co$ is said to be a Type I code.

To describe various properties (e.g. error-detection capability, error-correction capability, etc.) of a linear code, there are associated some specific polynomials with the code, which are called its weight enumerators. Below we define the complete weight enumerator of a linear code over $\z.$

The complete weight enumerator of a linear code $\co$ of length $n$ over $\z$ is defined as
\vspace{-2mm}$$cwe_{\co}(X_{\mu}:\mu\in\z)=\sum\limits_{c\in\co}X_0^{N_0(c)}X_1^{N_1(c)}\cdots X_{2^m-1}^{N_{2^m-1}(c)},\vspace{-2mm}$$ where for each $c \in \co,$ the number $N_{\mu}(c)$ $(0 \leq \mu \leq 2^m-1)$ equals the total number of components of $c$ that are equal to $\mu.$
Choie and Kim \cite{choie}  related complete weight enumerators of Type II codes over $\z$ with Jacobi forms, which are as defined below:
\begin{defin}\cite{eichler}
Let $\mathfrak{H}$ be the complex upper-half plane. A Jacobi form of weight $k$ and index $u$ ($k, u\in \mathbb{N}$) on the modular group $\Gamma(1)=SL(2,\mathbb{Z})$ is a holomorphic function $\phi:\mathfrak{H}\times \mathbb{C}\rightarrow \mathbb{C}$ satisfying the following:
\begin{enumerate}
\vspace{-2mm}\item[(1)] $(c\tau+d)^{-k}e^{-2\pi i u(\frac{cz^2}{c\tau+d})}\phi\Big(\frac{a\tau+b}{c\tau+d},\frac{z}{c\tau+d}\Big)=\phi(\tau,z)\text{ for all }\left(
           \begin{array}{cc}
           a & b \\
           c & d \\
            \end{array}
            \right)\in \Gamma(1),$
\vspace{-2mm}\item[(2)] $e^{2\pi i u(\lambda^2\tau+2\lambda z)}\phi(\tau,z+\lambda\tau+\epsilon)=\phi(\tau,z)\text{ for all }(\lambda,\epsilon)\in \mathbb{Z}^2,$ and \vspace{-2mm}\item[(3)] $\phi(\tau,z)$ has a Fourier series expansion of the form \vspace{-2mm}$$\sum\limits_{v=0}^{\infty}
\sum\limits_{\substack{r\in\mathbb{Z}\\ r^2\leq 4uv}}c(v,r)q^v\xi^r,\text{ where }q=e^{2\pi i\tau}\text{ and }\xi=e^{2\pi i z}.\vspace{-2mm}$$
\end{enumerate}
\end{defin}
One can also determine Jacobi forms from the shadow of a Type I code, which is as discussed below.
\subsection{Shadow of a Type I code}\label{secshadow}
Let $\co$ be a Type I code of length $n$ over $\z.$ Let us define  $\co_0=\{c \in \co: wt_E(c) \equiv 0~(\text{mod }2^{m+1})\}.$ It is easy to observe that $\co_0$ is a subcode of index 2 in $\co$ and  $\co_0$ is a subcode of index 4 in $\co_0^{\perp}.$ From this, it follows that $\co=\co_0\cup \co_2$ and $\co_0^{\perp}=\co_0\cup \co_2\cup \co_1\cup\co_3,$ where $\co_2=t+\co_0,$ $\co_1=s+\co_0$ and $\co_3=s+t+\co_0$ with $t\in\co\setminus\co_0$  and $s\in\co_0^{\perp}\setminus\co.$  Then the shadow of $\co$ is defined as the set $\mathcal{S}=\co_0^{\perp}\setminus\co=\co_1\cup\co_3.$

Now we have the following well-known result.
\begin{lem}\label{lemsglue}Let $\co$ be a Type I code of length $n$ over $\z.$ Then the integer \begin{itemize}\vspace{-2mm}\item[(a)] $n$ is even if and only if the glue group $\co_0^{\perp} /\co_0$ is the Klein 4-group.
\vspace{-2mm}\item[(b)] $n$ is odd if and only if the glue group $\co_0^{\perp}/\co_0$ is a cyclic group of order $4.$\end{itemize}\end{lem}
\begin{proof} For proof, see Proposition 3 of Dougherty et al. \cite{shadowz2k}.\end{proof}

Next we state orthogonality relations between the cosets of $\co_0$ in $\co_0^{\perp}.$
\begin{theo}\textup{\cite{shadowz2k}}\label{theosorel}
Let $\co$ be a Type I code of length $n$ over $\z$ with $\co_i$ $(0 \leq i \leq 3)$ as defined above.
\begin{itemize}
\vspace{-2mm}\item[(a)] If $n\equiv 2~(\text{mod }4),$ then we have the following:
\vspace{-2mm}\begin{center}
\begin{tabular}{|c|c|c|c|c|}
  \hline
   $\cdot$ & $\co_0$ & $\co_1$ & $\co_2$ & $\co_3$ \\
   \hline
  $\co_0$ & $\perp$ & $\perp$ & $\perp$ & $\perp$ \\
  $\co_1$ & $\perp$ & $\not\perp$ & $\not\perp$ & $\perp$ \\
  $\co_2$ & $\perp$ & $\not\perp$ & $\perp$ & $\not\perp$ \\
  $\co_3$ & $\perp$ & $\perp$ & $\not\perp$ & $\not\perp$ \\
  \hline
\end{tabular}\vspace{-2mm}
\end{center}
where for $0 \leq i,j \leq 3,$ the symbol $\perp$ in the $(i,j)$th position means $x\cdot y\equiv 0~(\text{mod }2^m)$ for each $x\in\co_i$ and $y\in\co_j,$ whereas the symbol $\not\perp$ means  $x\cdot y \equiv 2^{m-1}~(\text{mod }2^m)$ for each $x\in\co_i$ and $y\in\co_j.$
\vspace{-2mm}\item[(b)] If $n\equiv 0~(\text{mod }4),$ then we have the following:
\vspace{-2mm}\begin{center}
\begin{tabular}{|c|c|c|c|c|}
  \hline
   $\cdot$ & $\co_0$ & $\co_1$ & $\co_2$ & $\co_3$ \\
   \hline
  $\co_0$ & $\perp$ & $\perp$ & $\perp$ & $\perp$ \\
  $\co_1$ & $\perp$ & $\perp$ & $\not\perp$ & $\not\perp$ \\
  $\co_2$ & $\perp$ & $\not\perp$ & $\perp$ & $\not\perp$ \\
  $\co_3$ & $\perp$ & $\not\perp$ & $\not\perp$ & $\perp$ \\
  \hline
\end{tabular}\vspace{-2mm}
\end{center}
where for $0 \leq i,j \leq 3,$  the symbol $\perp$ in the $(i,j)$th position means $x\cdot y \equiv 0~(\text{mod }2^m)$ for each $x\in\co_i$ and $y\in\co_j,$ whereas the symbol $\not\perp$ means  $x\cdot y\equiv 2^{m-1} ~(\text{mod }2^m)$ for each $x\in\co_i$ and $y\in\co_j.$
\vspace{-2mm}\item[\textup{(c)}] Let $n\equiv a~(\text{mod }4),$ where $a=1$ or $3.$ Here we have the following:
\vspace{-2mm}\begin{center}
\begin{tabular}{|c|c|c|c|c|}
  \hline
   $\cdot$ & $\co_0$ & $\co_1$ & $\co_2$ & $\co_3$ \\
   \hline
  $\co_0$ & $0$ & $0$ & $0$ & $0$ \\
  $\co_1$ & $0$ & $2^{m-1}-2^{m-2}a$ & $2^{m-1}$ & $-2^{m-2}a$ \\
  $\co_2$ & $0$ & $2^{m-1}$ & $0$ & $2^{m-1}$ \\
  $\co_3$ & $0$ & $-2^{m-2}a$ & $2^{m-1}$ & $2^{m-1}-2^{m-2}a$ \\
  \hline
\end{tabular}\vspace{-2mm}
\end{center}
where for $0 \leq i,j \leq 3,$ the $(i,j)$th entry in the table represents the value of $x\cdot y$ $(x\in\co_i$ and $y\in\co_j)$ modulo $2^m.$
\end{itemize}
\end{theo}
\begin{proof}
For proof, see Dougherty et al. \cite[Theorem 6]{shadowz2k}.
\end{proof}
The following lemma provides the Euclidean weight (modulo $2^{m+1}$) of any vector in the shadow.
\begin{lem}\textup{\cite{shadowz2k}}\label{2.2}
Let $\co$ be a Type I code of length $n$ over $\z$ and $S=\co_1 \cup \co_3$ be its shadow code. Then for every $s \in S,$ we have $wt_E(s) \equiv 2^{m-2}n~(\text{mod }2^{m+1}).$
\end{lem}
\begin{proof}
For proof, see Dougherty et al. \cite[Theorem 5]{shadowz2k}.
\end{proof}
Note that the shadow is defined only for Type I codes, as the shadow of a Type II code is an empty set. This motivates Dougherty et al. \cite{shadowz4} to further generalize the notion of shadow for any self-dual linear code over $\mathbb{Z}_{4},$ which is as discussed below.
\subsection{Generalized shadow of a self-dual code}
Here we discuss the generalized shadow of a self-dual code of length $n$ over $\z.$

Let $\co$ be a self-dual code of length $n$ over $\z.$ For a vector $s\in\z^n \setminus\co,$ let us define a map $\psi_s:\co\rightarrow\z$ as $\psi_s(u)=u\cdot s$ for each $u\in\co.$ Note that $\psi_s$ is a group homomorphism from the additive group $\co$ to the additive group $\z$ and the set $\co_0=ker(\psi_s)$ is a proper subcode of $\co.$  Hence by the first isomorphism theorem, we have $\co/\co_0\simeq Im (\psi_s).$ This gives $[\co:\co_0]=|Im (\psi_s)|=r,$ where $r>1$ is a divisor of $2^m.$ Further choose a vector $s$ such that $Im (\psi_s)=\{0,2^{m-1}\}.$ With this choice of $s,$ we have $[\co:\co_0]=2$ and we write $\co=\co_0\cup\co_2.$ Also note that $[\co_0^{\perp}:\co_0]=4$ and let us write $\co_0^{\perp}=\co_0\cup\co_2\cup\co_1\cup\co_3.$ Then the generalized shadow of $\co$ with respect to the vector $s$ is defined as $\mathcal{S}_g(s)=\co_1\cup\co_3.$

Now we state some important properties of the generalized shadow of $\co$ with respect to the vector $s.$
\begin{lem}
Let $\co$ be a self-dual code of length $n$ over $\z.$  Let  $s \in \z^n\setminus \co$ be such that $Im (\psi_s)=\{0,2^{m-1}\}.$ With respect to the vector $s,$ let $\co=\co_0\cup\co_2$ and  $\co_0^{\perp}=\co_0\cup\co_2\cup\co_1\cup\co_3.$ Then we have $\co_1=s+\co_0,$ $\co_2=t+\co_0,$ $\co_3=s+t+\co_0$ for some $t\in\co\setminus\co_0.$
\end{lem}
\begin{proof} Working in a similar way as in Lemma 3.12 of Dougherty et al. \cite{shadowz4}, the result follows.
\end{proof}
\begin{lem}
Let $\co$ be a self-dual code of length $n$ over $\z.$ Let $s \in \z^n \setminus \co$ be such that $Im (\psi_s)=\{0,2^{m-1}\}.$ Let $\co=\co_0\cup\co_2$ and  $\co_0^{\perp}=\co_0\cup\co_2\cup\co_1\cup\co_3$  with respect to the vector $s.$ Then we have the following:
\begin{itemize}
\vspace{-2mm}\item[(a)] If $s$ is a vector in $\z^n \setminus \co$ satisfying $s\cdot s\equiv 0~(\text{mod }2^m),$ then the orthogonality relations between the cosets of $\co_0$ in $\co_0^{\perp}$ are given by
    \vspace{-2mm}\begin{center}
\begin{tabular}{|c|c|c|c|c|}
  \hline
   $\cdot$ & $\co_0$ & $\co_1$ & $\co_2$ & $\co_3$ \\
   \hline
  $\co_0$ & $\perp$ & $\perp$ & $\perp$ & $\perp$ \\
  $\co_1$ & $\perp$ & $\perp$ & $\not\perp$ & $\not\perp$ \\
  $\co_2$ & $\perp$ & $\not\perp$ & $\perp$ & $\not\perp$ \\
  $\co_3$ & $\perp$ & $\not\perp$ & $\not\perp$ & $\perp$ \\
  \hline
\end{tabular}\vspace{-2mm}
\end{center}
where the symbol $\perp$ means $x\cdot y\equiv 0~(\text{mod }2^m)$ and the symbol $\not\perp$ means $x\cdot  y\equiv 2^{m-1}~(\text{mod }2^m).$
\vspace{-2mm}\item[(b)] If $s$ is a vector in $\z^n \setminus \co$ satisfying $s\cdot s\equiv 2^{m-1}~(\text{mod }2^m),$ then the orthogonality relations between the cosets of $\co_0$ in $\co_0^{\perp}$ are given by \vspace{-2mm}\begin{center}
\begin{tabular}{|c|c|c|c|c|}
  \hline
   $\cdot$ & $\co_0$ & $\co_1$ & $\co_2$ & $\co_3$ \\
   \hline
  $\co_0$ & $\perp$ & $\perp$ & $\perp$ & $\perp$ \\
  $\co_1$ & $\perp$ & $\not\perp$ & $\not\perp$ & $\perp$ \\
  $\co_2$ & $\perp$ & $\not\perp$ & $\perp$ & $\not\perp$ \\
  $\co_3$ & $\perp$ & $\perp$ & $\not\perp$ & $\not\perp$ \\
  \hline
\end{tabular}\vspace{-2mm}
\end{center}
where the symbol $\perp$ means $x\cdot y\equiv 0~(\text{mod }2^m)$ and the symbol $\not\perp$ means $x\cdot  y\equiv 2^{m-1}~(\text{mod }2^m).$
\end{itemize}
\end{lem}
\begin{proof} Working in a similar way as in Lemma 3.13 of Dougherty et al. \cite{shadowz4}, the result follows.
\end{proof}
In the following sections, we will construct self-dual codes (of higher lengths) over $\z$ using shadows of Type I codes over $\z$ and generalized shadows of self-dual codes over $\z$ with respect to a vector $s \in \z^n,$ which is not a codeword and satisfies $s \cdot s \equiv 0 \text{ or }2^{m-1}~(\text{mod }2^m).$ While doing so, we will generalize the construction method employed by Brualdi and Pless \cite{bru} and Tsai \cite{han}, besides rectifying an error in the construction proposed by Dougherty et al. \cite{shadowz4}.
\section{Construction of Type I and Type II codes from shadows of Type I codes}\label{secconsshad}
 In  order to construct Type I and Type II codes  over the ring $\z$ ($m \geq 1$ is an integer) from the shadow of a Type I code over $\z,$ we will first construct a self-orthogonal code over $\z$ from a self-dual code over $\z.$  For this, we need the following notations.

Throughout this paper, let $\left<A\right>$ denote the vector subspace of $\z^n$ generated by a subset $A$ of $\z^n$ and  $o(a)$ denote the additive order of an element $a\in\z^k.$ Let us define a function $\eta : \z \times \z \rightarrow \{0,1,2,3\}$ as
\vspace{-2mm}\begin{equation}\label{n}\eta(i,j)=\left\{
              \begin{array}{ll}
                \left[i\right]_2+2\left[j\right]_2 & \text{if }\co_0^{\perp}/\co_0
                \text{ is the Klein 4-group;}\\
                \left[i+2j\right]_4 & \text{if }\co_0^{\perp}/\co_0 \text{ is a cyclic group of order 4,}
              \end{array}
            \right.\end{equation} where for integers $a$ and $r \geq 1,$ $[a]_r$ denotes the remainder obtained upon dividing $a$ by $r.$ Then we observe the following:
\begin{lem}\label{c} For $1 \leq i,j \leq 2^m-1,$ we have $i\co_1+j\co_2=\co_{\eta(i,j)}.$\end{lem}\begin{proof}Proof is trivial.\end{proof}

In the following proposition, we construct a self-orthogonal code (of higher length) over $\z$  from a self-dual code of length $n$ over $\z.$
\begin{prop}\label{p}
Let $\co$ be a self-dual code of length $n$ over $\z$ and $\co_0$ be a proper subcode of index 2 in $\co.$ Let $\co=\co_0\cup\co_2$ and $\co_0^{\perp}=\co_0\cup\co_2\cup\co_1\cup\co_3,$ where $\co_2=t+\co_0,$ $\co_1=s+\co_0$ and $\co_3=s+t+\co_0$ for some $t\in\co\setminus\co_0$ and  $s\in\co_0^{\perp}\setminus\co.$  Suppose that there exist vectors $v_1,v_2 \in \z^{k}$ satisfying the following three properties:
 \vspace{-2mm}\begin{enumerate}\item[$(\mathbf{P_1)}$] If there exist $\alpha,\beta \in \z$ ($0 \leq \alpha < o(v_1)$ and $0 \leq \beta < o(v_2)$) such that $\alpha v_1+\beta v_2=0,$ then $\alpha=\beta=0.$
 \item[$(\mathbf{P_2)}$] $\label{key}
v_{1}\cdot v_{1}\equiv -s\cdot s~(\text{mod }2^m),~~~
v_{1}\cdot v_{2}\equiv  -t\cdot s~(\text{mod }2^m)\text{~~ and ~~} v_{2}\cdot v_{2}\equiv  -t\cdot t~(\text{mod }2^m).$\item[$(\mathbf{P_3)}$] $o(v_1)\equiv \left\{\begin{array}{ll}
                  0~(\text{mod }2) & \text{if } \co_0^{\perp}/\co_0 \text{ is the Klein 4-group;} \\
                  0~(\text{mod }4) & \text{if } \co_0^{\perp}/\co_0 \text{ is a cyclic group of order 4}
                \end{array}\right.  \text{ and }
o(v_2)\equiv 0~(\text{mod }2)$ provided $m \geq 2.$\end{enumerate}
 Here the positive integer $k$ is to be chosen suitably depending upon the existence of vectors $v_1$ and $v_2$ in $\z^k.$ Now if $(v_i,\co_i)=\left\{\left(v_{i},c_{i}\right): c_{i} \in \co_{i}\right\}$ for $1 \leq i \leq 2,$ then the set  \vspace{-2mm}$$\co^*=
\langle (v_1,\co_1) \cup (v_2,\co_2) \rangle=\{(iv_1+jv_2,c_{ij}): 1 \leq i \leq o(v_1), ~1 \leq j \leq o(v_2), c_{ij} \in \co_{\eta(i,j)}\}\vspace{-2mm}$$ is a self-orthogonal code of length $n+k$ over $\z.$
Moreover, we have \vspace{-2mm}$$|\co^*|=o(v_1)o(v_2)|\co_0|=\frac{1}{2}o(v_1)o(v_2)|\co|=o(v_1)o(v_2)2^{\frac{mn}{2}-1}.\vspace{-2mm}$$
\end{prop}
\begin{proof} First of all, we will show that the code $\co^*$ is a self-orthogonal code.
For this, we see that it is enough to prove that \vspace{-2mm}\begin{equation}\label{k}(v_1,c_1)\cdot (v_1,c_1)=(v_1,c_1)\cdot (v_2,c_2)=(v_2,c_2)\cdot (v_2,c_2)=0 \end{equation}for all $c_1 \in \co_1$ and $c_2 \in \co_2.$

To prove this, let $c_1 \in \co_1$ and $c_2 \in \co_2$ be arbitrarily fixed. Since $\co_1=s+\co_0$ and $\co_2=t+\co_0,$ we write $c_1=s+c_0$ and $c_2 =t+c'_0$ for some $c_0, ~c'_0 \in \co_0.$ Also $\co_0^{\perp}=\co_0 \cup \co_1 \cup \co_2 \cup \co_3,$ $s \in \co_1$ and $t \in \co_2$  imply that $c_0 \cdot c_0\equiv c'_0 \cdot c'_0 \equiv c_0 \cdot c'_0\equiv c_0 \cdot t \equiv c'_0 \cdot t \equiv c_0 \cdot s \equiv c'_0 \cdot s\equiv 0~(\text{mod }2^m).$ This gives $
c_1\cdot c_1\equiv s\cdot s~(\text{mod }2^m),~c_1\cdot c_2\equiv s\cdot t~(\text{mod }2^m)$ and $c_2\cdot c_2\equiv t\cdot t~(\text{mod }2^m).$  Now as the vectors $v_1$ and $v_2$ satisfy the property $(\mathbf{P}_2)$, we have
$v_{1}\cdot v_{1}\equiv -s\cdot s~(\text{mod }2^m),~
v_{1}\cdot v_{2}\equiv  -t\cdot s~(\text{mod }2^m)$ and $ v_{2}\cdot v_{2}\equiv  -t\cdot t~(\text{mod }2^m).$
From this, we obtain $v_{1}\cdot v_{1}\equiv -c_1\cdot c_1~(\text{mod }2^m),~
v_{1}\cdot v_{2}\equiv  -c_2\cdot c_1~(\text{mod }2^m)$ and $ v_{2}\cdot v_{2}\equiv  -c_2\cdot c_2~(\text{mod }2^m),$ which implies \eqref{k}.

We next assert that \vspace{-4mm}$$\co^*=\{(iv_1+jv_2,c_{ij}):1\leq i\leq o(v_1),~1\leq j\leq o(v_2),~c_{ij} \in \co_{\eta(i,j)}\}=\displaystyle\bigcup_{i=1}^{o(v_1)}
\bigcup_{j=1}^{o(v_2)}(iv_1+jv_2, \co_{\eta(i,j)}),\vspace{-2mm}$$ where $(iv_1+jv_2, \co_{\eta(i,j)})= \{(iv_1+jv_2,c_{ij}): c_{ij} \in \co_{\eta(i,j)}\}$ for each $i$ and $j.$

To prove this assertion, we first note that \vspace{-2mm}$$o(\co_1)=\left\{\begin{array}{ll} 2 & \text{if }\co_0^{\perp}/\co_0 \text{ is the Klein 4-group;}\\ 4 & \text{if }\co_0^{\perp}/\co_0 \text{ is a cyclic group of order 4} \end{array}\right.\vspace{-2mm}$$ and  $o(\co_2)=2,$ where $o(\co_i)$ ($i=1,2$) denotes the order of $\co_i \in \co_0^{\perp}/\co_0.$ Now as the vectors $v_1$ and $v_2$ satisfy the property ($\mathbf{P}_2$), we see that $o(\co_1)$ divides $o(v_1)$ and $o(\co_2)$ divides $o(v_2),$ which gives $o(v_1)\co_1=\co_0$ and $o(v_2)\co_2=\co_0.$ Also note that $\co_1+\co_0=\co_1$ and $\co_2+\co_0=\co_2.$ From this, it follows that for every $i \equiv i_1~(\text{mod }o(v_1))$ and $j \equiv j_1 ~(\text{mod }o(v_2))$ with $1 \leq i_1 \leq o(v_1)$ and $1 \leq j_1 \leq o(v_2),$ we have $i(v_1,\co_1)=i_1(v_1,\co_1)$ and $j(v_2,\co_2)=j_1(v_2,\co_2).$ Therefore the set $\co^*$ reduces to $\co^*=\{i(v_1,c_1)+j(v_2,c_2):1\leq i\leq o(v_1),~1\leq j\leq o(v_2),~c_1\in\co_1,~c_2\in\co_2\}.$ Also by Lemma \ref{c},  for every $c_1 \in \co_1$ and $c_2 \in \co_2,$ we note that $ic_1 +jc_2 \in \co_{\eta(i,j)}$ for each $i$ and $j,$ which proves the assertion.

Next to prove that $|\co^*|=o(v_1)o(v_2)|\co_0|=o(v_1)o(v_2)2^{\frac{mn}{2}-1},$ we need to show that the sets   $ (iv_1+jv_2,\co_{\eta_{i,j}})$ ($1 \leq i \leq o(v_1),$ $1 \leq j \leq o(v_2)$) are distinct (and hence disjoint) in $\z^{n+k},$ as $|\co_{\eta(i,j)}|=|\co_0|=\frac{1}{2}|\co|=2^{\frac{mn}{2}-1}$ for each $i$ and $j.$ To prove this, we
consider the following three cases separately:
\textbf{(i)} $i=o(v_1)$ and $1 \leq j \leq o(v_2),$ \textbf{(ii)} $1 \leq i \leq o(v_1)-1$ and $j=o(v_2),$ \textbf{(iii)} $1 \leq i \leq o(v_1)-1$ and $1 \leq j \leq o(v_2)-1.$
\begin{description}
  \item[Case (i):] Let $i=o(v_1)$ and $1\leq j\leq o(v_2).$ Here we have $(iv_1+jv_2,i\co_1+j\co_2)=(jv_2,i\co_1+j\co_2).$ Working as earlier, we see that $o(\co_1)$ divides $o(v_1),$ which gives $o(v_1)\co_1=\co_0.$ This implies that $(iv_1+jv_2,i\co_1+j\co_2)=(jv_2,j\co_2),$ as $\co_0+j\co_2=j\co_2$ for each $j.$ Further it is easy to see that the sets $(jv_2,j\co_2)$ ($1 \leq j \leq o(v_2)$) are disjoint in $\z^{n+k}.$
  \item[Case (ii):] Let $1\leq i\leq o(v_1)-1$ and $j=o(v_2).$ Here we have $(iv_1+jv_2,i\co_1+j\co_2)=(iv_1,i\co_1+j\co_2).$ As the vector $v_2$ satisfies the property $(\mathbf{P_3}),$  $o(\co_2)=2$ divides $o(v_2).$ This gives $o(v_2)\co_2=\co_0,$ which implies that $(iv_1+jv_2,i\co_1+j\co_2)=(iv_1,i\co_1),$ as $i\co_1+\co_0=i\co_1$ for each $i.$ Also it is easy to see that the sets $(iv_1,i\co_1)$ ($1\leq i\leq o(v_1)-1$) are disjoint in $\z^{n+k}.$
  \item[Case (iii):]  Let $1\leq i\leq o(v_1)-1$ and $1\leq j\leq o(v_2)-1.$ Here we need to show that the sets $(iv_1+jv_2,\co_{\eta(i,j)})~$ ($1 \leq i \leq o(v_1)-1$ and $1 \leq j \leq o(v_2)-1$) are disjoint in $\z^{n+k}.$ For this, suppose (if possible) that  $(iv_1+jv_2,\co_{\eta(i,j)})=(i'v_1+j'v_2,\co_{\eta(i'j')})$ holds for some $1\leq i'\leq o(v_1)-1$ and $1\leq j'\leq o(v_2)-1.$ This gives $((i-i')v_1+(j-j')v_2,\co_{\eta(i,j)}-\co_{\eta(i',j')})=(\mathbf{0},\co_0),$ which gives $(i-i')v_1+(j-j')v_2=0.$  This implies that $r_1v_1+r_2v_2=0,$ where $\left[i-i'\right]_{o(v_1)}=r_1$ and $\left[j-j'\right]_{o(v_2)}=r_2.$ As the vectors $v_1$ and $v_2$ satisfy the property $(\mathbf{P_1}),$ we have $r_1=0$ and $r_2=0.$ Further since $1 \leq i,i' \leq o(v_1)-1$ and $1 \leq j, j' \leq o(v_2)-1,$ we must have $r_1=|i-i'|$ and $r_2=|j-j'|,$ which gives $i=i'$ and $j=j'.$ From this, it follows that the sets $(iv_1+jv_2,\co_{\eta(i,j)})~$ ($1\leq i\leq o(v_1)-1$ and $1\leq j\leq o(v_2)-1$) are disjoint in $\z^{n+k}.$
\end{description}
This completes the proof of the proposition.\end{proof}
\begin{rem} When $m=1,$ the binary code $\co^*$ (if it exists) is given by $\co^*=\left(\textbf{0},\co_0\right)\cup \left(v_1,\co_1 \right) \cup \left(v_2,\co_2\right) \cup \left(v_1+v_2,\co_3\right),$ as $\co_1+\co_2=\co_3.$ (Note that $(v_1+v_2,\co_3)=\{(v_1+v_2,c_3): c_3 \in \co_3\}.$)
\end{rem}
From now onwards, we will follow the same notations as in Proposition \ref{p}, and we will consider the subscripts of $X$'s modulo $2^m.$

Next by Lemma \ref{a}, we see that self-dual codes of even length $n$ over $\z$ exist for all $m \geq 1,$ and self-dual codes of odd length $n$ over $\z$ exist only when $m$ is even.

In the following theorem, we construct Type I and Type II codes (of higher lengths) over $\z$ from Type I codes of even length $n$ over $\z$ for all $m \geq 1,$ and also determine their complete weight enumerators.
\begin{theo}\label{theosneven}
Let $\co$ be a Type I code of even length $n$ over $\z$ ($m \geq 1$ is an integer) and $\co_0=\{c \in \co: wt_E(c) \equiv 0~(\text{mod }2^{m+1})\}.$ Here we have $\eta(i,j)=[i]_2+2[j]_2$ for $0 \leq i,j \leq 2^m-1,$ and  the following hold:
\begin{description}
 \vspace{-2mm} \item[I.] Let $n\equiv 2~(\text{mod }4).$
  \begin{itemize}
 \vspace{-2mm} \item[(a)] Let $v_{1},v_{2} \in \z^2$  be chosen as
  \vspace{-2mm}\begin{eqnarray*}
  \begin{array}{ll}
 v_{1}=\left\{\begin{array}{ll}(2^{\f-1},2^{\f-1}) & \text{if } m \text{ is even;}\\
  (2^{\fo},0) & \text{if } m \text{ is odd}\end{array} \right. & \text{ and ~~~}
    v_{2}=\left\{\begin{array}{ll}(2^{\f},0) & \text{if } m \text{ is even;}\\
  (2^{\fo},2^{\fo}) & \text{if } m \text{ is odd.}\end{array} \right.
  \end{array}
  \end{eqnarray*}
Then $\co^*$ is a self-dual code of length $n+2.$ Furthermore, $\co^*$ is a Type II code when $n\equiv 6~(\text{mod }8)$  and $\co^*$ is a Type I code when $n\equiv 2~(\text{mod }8).$ The complete weight enumerator $cwe_{\co^*}(X_\mu:\mu\in\z)$ of $\co^*$ is given by
\small     \vspace{-2mm} $$cwe_{\co^*}(X_\mu:\mu\in\z)=\left\{\begin{array}{ll}\sum\limits_{i=1}^{2^{\f+1}}\sum\limits_{j=1}^{2^{\f}}
      X_{(i+2j)2^{\f-1}}X_{i2^{\f-1}}cwe_{\co_{\eta(i,j)}}(X_\mu:\mu\in\z)& \text{if }m\text{ is even;}\vspace{2mm}\\\sum\limits_{i=1}^{2^{\fe}}
      \sum\limits_{j=1}^{2^{\fe}}X_{(i+j)2^{\fo}}
      X_{j2^{\fo}}cwe_{\co_{\eta(i,j)}}(X_\mu:\mu\in\z)&
      \text{if }m \text{ is odd.}\end{array}\right.\vspace{-2mm}$$
 \normalsize
 \vspace{-2mm}\item[(b)] Let $v_{1},v_{2} \in \z^6$ and $w_1,w_2,w_3,w_4 \in \z^{n+6}$ be chosen as \vspace{-2mm}\begin{eqnarray*}
 v_{1}&=&\left\{\begin{array}{ll}(2^{\f-1},2^{\f-1},2^{\f-1},2^{\f-1},2^{\f-1},2^{\f-1}) & \text{if } m \text{ is even;} \\(2^{\fo},0,2^{\fo},0,2^{\fo},0) & \text{if }m \text{ is odd,} \end{array}\right.\allowdisplaybreaks\\
 v_{2}&=&\left\{\begin{array}{ll}(2^{\f},0,0,0,0,0) & \text{if } m \text{ is even;} \\(2^{\fo},2^{\fo},0,0,0,0) & \text{if }m \text{ is odd,} \end{array}\right.\\
 w_1&=&\left\{\begin{array}{ll}(2^{\f},0,2^{\f},0,0,0,0,\cdots,0) & \text{if } m \text{ is even;} \\(2^{\fo},2^{\fo},2^{\fo},2^{\fo},0,0,0,\cdots,0) & \text{if }m \text{ is odd,} \end{array}\right.\\
 w_2&=&\left\{\begin{array}{ll}(0,2^{\f},0,2^{\f},0,0,0,\cdots,0) & \text{if } m \text{ is even;} \\(2^{\fe},0,0,0,0,0,0,\cdots,0) & \text{if }m \text{ is odd,} \end{array}\right.\\
  w_3&=&\left\{\begin{array}{ll}(0,0,2^{\f},0,2^{\f},0,0,\cdots,0) & \text{if } m \text{ is even;} \\(0,0,2^{\fo},2^{\fo},2^{\fo},2^{\fo},0, \cdots,0) & \text{if }m \text{ is odd,} \end{array}\right.\\
w_4&=&\left\{\begin{array}{ll}(0,0,0,2^{\f},0,2^{\f},0,\cdots,0) & \text{if } m \text{ is even;} \\(0,2^{\fe},0,0,0,0,0,\cdots,0) & \text{if }m \text{ is odd.} \end{array}\right.\end{eqnarray*}
 Then the code $\co'=\left< \co^* \cup \{w_1,w_2,w_3,w_4\}\right>$ is a self-dual code of length $n+6.$ Furthermore, the code $\co'$ is a Type II code when $n\equiv 2~(\text{mod }8)$ and $\co'$ is a Type I code when $n\equiv 6~(\text{mod }8).$ The complete weight enumerator $cwe_{\co'}(X_\mu:\mu\in\z)$ of $\co'$ is given by \small
\vspace{-2mm}$$cwe_{\co'}(X_\mu:\mu\in\z)=\left\{\begin{array}{ll}\sum_1 X_{(i+2j+2k_1)2^{\f-1}}X_{(i+2k_2)2^{\f-1}}
      X_{(i+2(k_1+k_3))2^{\f-1}}\\X_{(i+2(k_2+k_4))2^{\f-1}} X_{(i+2k_3)2^{\f-1}}X_{(i+2k_4)2^{\f-1}}\\cwe_{\co_{\eta(i,j)}}(X_\mu:\mu\in\z) & \text{if } m \text{ is even;}\vspace{2mm}\\
      \sum_2 X_{(i+j+k_1+2k_2)2^{\fo}}X_{(j+k_1+2k_4)2^{\fo}}
      X_{(i+k_1+k_3)2^{\fo}}\\X_{(k_1+k_3)2^{\fo}} X_{(i+k_3)2^{\fo}}X_{k_32^{\fo}}cwe_{\co_{\eta(i,j)}}(X_\mu:\mu\in\z) &\text{if }m\text{ is odd,}\end{array}\right.\vspace{-1mm}$$ \normalsize where the summation $\sum_1$ runs over all integral 6-tuples $(i,j,k_1,k_2,k_3,k_4)$ satisfying $1\leq j,k_1,k_2,\\k_3,k_4\leq {2^{\f}}$ and $1\leq i\leq {2^{\f+1}},$ whereas the summation $\sum_2$ runs over all integral 6-tuples $(i,j,k_1,k_2,k_3,k_4)$ satisfying $1\leq i,j,k_1,k_{3}\leq {2^{\fe}}$ and $1\leq k_2,k_{4}\leq {2^{\fo}}.$
\end{itemize}
\vspace{-2mm}\item[II.] Let $n\equiv 0~(\text{mod }4).$
   \begin{itemize}
  \vspace{-2mm}\item[(a)] Let $v_{1},v_{2} \in \z^4$ and $w_1,w_2 \in \z^{n+4}$ be chosen as \vspace{-2mm}\begin{eqnarray*}
 v_{1}&=&\left\{\begin{array}{ll}(2^{\f-1},2^{\f-1},2^{\f-1},2^{\f-1}) & \text{if } m \text{ is even;} \\(2^{\fo},0,2^{\fo},0) & \text{if }m \text{ is odd,} \end{array}\right.\allowdisplaybreaks\\
 v_{2}&=&\left\{\begin{array}{ll}(2^{\f},0,0,0) & \text{if } m \text{ is even;} \\(2^{\fo},2^{\fo},0,0) & \text{if }m \text{ is odd,} \end{array}\right.\allowdisplaybreaks\\
 w_1&=&\left\{\begin{array}{ll}(2^{\f},2^{\f},0,0,0,\cdots,0) & \text{if } m \text{ is even;} \\(2^{\fo},2^{\fo},2^{\fo},2^{\fo},0,\cdots,0) & \text{if }m \text{ is odd,} \end{array}\right.\\
 w_2&=&\left\{\begin{array}{ll}(0,0,2^{\f},2^{\f},0,\cdots,0) & \text{if } m \text{ is even;} \\(2^{\fe},0,0,0,0,\cdots,0) & \text{if }m \text{ is odd.} \end{array}\right.
 \end{eqnarray*} Then $\co'=\left< \co^* \cup \{w_1,w_2\}\right>$ is a self-dual code of length $n+4$ over $\z.$ Furthermore, $\co'$ is a Type II code when $n\equiv 4~(\text{mod }8)$ and $\co'$ is a Type I code when $n\equiv 0~(\text{mod }8).$ The complete weight enumerator $cwe_{\co'}(X_\mu:\mu\in\z)$ of $\co'$ is given by
\vspace{-2mm}$$cwe_{\co'}(X_\mu:\mu\in\z)=\left\{\begin{array}{ll}\sum_1 X_{(i+2j+2k_1)2^{\f-1}}X_{(i+2k_1)2^{\f-1}}\\
      X_{(i+2k_2)2^{\f-1}}^2cwe_{\co_{\eta(i,j)}}(X_\mu:\mu\in\z)& \text{if }m \text{ is even;}\vspace{2mm}\\\sum_2
      X_{(i+j+k_1+2k_2)2^{\fo}}X_{(j+k_1)2^{\fo}}
      X_{(i+k_1)2^{\fo}}\\X_{k_12^{\fo}}cwe_{\co_{\eta(i,j)}}(X_\mu:\mu\in\z) & \text{if }m \text{ is odd,}\end{array}\right.\vspace{-2mm}$$
 where the summation $\sum_1$ runs over all integral 4-tuples $(i,j,k_1,k_2)$ satisfying $1\leq j,k_1,k_2\leq 2^{\f}$ and $1\leq i\leq 2^{\f+1},$ whereas the summation $\sum_2$ runs over all integral 4-tuples $(i,j,k_1,k_2)$ satisfying $1\leq i,j,k_1\leq {2^{\fe}}$ and $1\leq k_2\leq {2^{\fo}}.$
   \vspace{-1mm}\item[(b)] Let $v_{1},v_{2} \in \z^8$ and $w_1,w_2,w_3,w_4,w_5,w_6 \in \z^{n+8}$ be chosen as
  \vspace{-2mm}\begin{eqnarray*}
 v_{1}&=&\left\{\begin{array}{ll}(2^{\f-1},2^{\f-1},2^{\f-1},2^{\f-1},2^{\f-1},
  2^{\f-1},2^{\f-1},2^{\f-1}) & \text{if } m \text{ is even;} \\(2^{\fo},0,2^{\fo},0,2^{\fo},0,2^{\fo},0) & \text{if }m \text{ is odd,} \end{array}\right.\allowdisplaybreaks\\
  v_{2}&=&\left\{\begin{array}{ll}(2^{\f},0,0,0,0,0,0,0)& \text{if } m \text{ is even;} \\(2^{\fo},2^{\fo},0,0,0,0,0,0) & \text{if }m \text{ is odd,} \end{array}\right.\allowdisplaybreaks\\
 w_1&=&\left\{\begin{array}{ll}(2^{\f},0, 2^{\f},0,\cdots,0) & \text{if } m \text{ is even;} \\(2^{\fo},2^{\fo},2^{\fo},2^{\fo},0,\cdots,0) & \text{if }m \text{ is odd,} \end{array}\right.\\
 w_2&=&\left\{\begin{array}{ll}(0,2^{\f},0, 2^{\f},0,\cdots,0) & \text{if } m \text{ is even;} \\(0,0,2^{\fo},2^{\fo},2^{\fo},2^{\fo},0,\cdots,0) & \text{if }m \text{ is odd,} \end{array}\right.\allowdisplaybreaks\\
 w_3&=&\left\{\begin{array}{ll} (0,0,2^{\f},0,2^{\f},0,\cdots,0) & \text{if } m \text{ is even;} \\(0,0,0,0,2^{\fo},2^{\fo},2^{\fo},2^{\fo},0,\cdots,0) & \text{if }m \text{ is odd,} \end{array}\right.\allowdisplaybreaks\\
 w_4&=&\left\{\begin{array}{ll}(0,0,0,2^{\f},0, 2^{\f},0,\cdots,0) & \text{if } m \text{ is even;} \\(2^{\fe},0,0,\cdots,0) & \text{if }m \text{ is odd,} \end{array}\right.\allowdisplaybreaks\\
 w_5&=&\left\{\begin{array}{ll}(0,0,0,0,2^{\f},0, 2^{\f},0,\cdots,0) & \text{if } m \text{ is even;} \\(0,2^{\fe},0,0,\cdots,0) & \text{if }m \text{ is odd,} \end{array}\right.\\
 w_6&=&\left\{\begin{array}{ll}(0,0,0,0,0,2^{\f},0, 2^{\f},0,\cdots,0) & \text{if } m \text{ is even;} \\(0,0,2^{\fe},0,0,\cdots,0) & \text{if }m \text{ is odd.} \end{array}\right.
 \end{eqnarray*}
 Then $\co'=\left<\co^* \cup \{w_1,w_2,w_3,w_4,w_5,w_6\} \right>$ is a self-dual code of length $n+8$ over $\z.$ Furthermore, $\co'$ is a Type II code of length $n+8$ over $\z$ when $n\equiv 0~(\text{mod }8)$ and $\co'$ is a Type I code when $n\equiv 4~(\text{mod }8).$ The complete weight enumerator $cwe_{\co'}(X_\mu:\mu\in\z)$ of $\co'$ is given by \small
      \vspace{-2mm}\begin{eqnarray*}
      cwe_{\co'}(X_\mu:\mu\in\z)=\left\{\begin{array}{ll}\sum_1 X_{(i+2j+2k_1)2^{\f-1}}X_{(i+2k_2)2^{\f-1}}
      X_{(i+2(k_1+k_3))2^{\f-1}}\\X_{(i+2(k_2+k_4))2^{\f-1}}X_{(i+2(k_3+k_5))2^{\f-1}}X_{(i+2(k_4+k_6))2^{\f-1}}\\ X_{(i+2k_5)2^{\f-1}}X_{(i+2k_6)2^{\f-1}} cwe_{\co_{\eta(i,j)}}(X_\mu:\mu\in\z) & \text{if } m \text{ is even;}\vspace{2mm}\\\sum_2 X_{(i+j+k_1+2k_4)2^{\fo}}X_{(j+k_1+2k_5)2^{\fo}}
      X_{(i+k_1+k_2+2k_6)2^{\fo}}
      \\X_{(k_1+k_2)2^{\fo}}X_{(i+k_2+k_3)2^{\fo}}
      X_{(k_2+k_3)2^{\fo}}X_{(i+k_3)2^{\fo}}\\
      X_{k_32^{\fo}}cwe_{\co_{\eta(i,j)}}(X_\mu:\mu\in\z) & \text{if }m \text{ is odd,}\end{array}\right.\end{eqnarray*}\normalsize
 where the summation $\sum_1$ runs over all integral 8-tuples $(i,j,k_1,k_2,k_3,k_4,k_5,k_6)$ satisfying $1\leq j,k_1,k_2,k_3,k_4,k_5,k_6\leq {2^{\f}}$ and $1\leq i\leq {2^{\f+1}},$ whereas the summation $\sum_2$ runs over all integral 8-tuples $(i,j,k_1,k_2,k_3,k_4,k_5,k_6)$ satisfying $1\leq i,j,k_1,k_2,k_{3}\leq {2^{\fe}}$ and $1\leq k_4,k_{5},k_6\leq {2^{\fo}}.$
\end{itemize}
\end{description}
\end{theo}
\begin{proof} As $n$ is even, by Lemma \ref{lemsglue}(a), we see that the glue group $\co_0^{\perp}/\co_0$ is the Klein 4-group. Thus by \eqref{n}, we have $\eta(i,j)= \left[i\right]_2+2\left[j\right]_2$ for each $i$ and $j.$ Here we will apply Proposition \ref{p} to construct a self-dual code in each case.
\\{\bf I.} Let $n\equiv 2~(\text{mod }4).$ First of all, we will construct a self-orthogonal code $\co^*$ by applying Proposition \ref{p}. For this, we need to choose a suitable positive integer $k$ and vectors $v_{1}$, $v_{2}$ in $\z^k$ satisfying the properties $(\mathbf{P_1}), (\mathbf{P_2})$ and $(\mathbf{P_3}).$ Now as $s\in\co_1$ and $t\in\co_2,$ by Theorem \ref{theosorel}(a), we have $s\cdot s\equiv 2^{m-1}~(\text{mod }2^m),~s\cdot t\equiv 2^{m-1}~(\text{mod }2^m)$ and $t\cdot t\equiv  0~(\text{mod }2^m).$ Therefore $(\mathbf{P_2})$ becomes
\vspace{-2mm}\begin{eqnarray}\label{eq1}
v_{1}\cdot v_{1} \equiv 2^{m-1}~(\text{mod }2^m),&
v_{1}\cdot v_{2}\equiv  2^{m-1}~(\text{mod }2^m), &
v_{2}\cdot v_{2}\equiv  0~(\text{mod }2^m).
\vspace{-2mm}\end{eqnarray}
\begin{itemize}
\item[(a)] It is clear that  vectors $v_{1}$ and $v_{2}$ satisfy the property $(\mathbf{P_1})$ and congruences \eqref{eq1}. Further, since \begin{eqnarray*}o(v_1)=\left\{\begin{array}{ll}2^{\frac{m}{2}+1}& \text{if }m \text{ is even;}\\2^{\frac{m+1}{2}} & \text{if }m \text{ is odd,} \end{array}\right. & o(v_2)=\left\{\begin{array}{ll}2^{\frac{m}{2}} & \text{if }m \text{ is even;}\\2^{\frac{m+1}{2}} & \text{if }m \text{ is odd}\end{array}\right.\end{eqnarray*} and $\co_0^{\perp}/\co_0$ is the Klein 4-group by Lemma \ref{lemsglue}(a), we see that vectors $v_1$ and $v_2$ also satisfy the property $(\mathbf{P_3}).$ Therefore by Proposition \ref{p}, the code $\co^*$ is a self-orthogonal code of length $n+2$ over $\z.$ As $o(v_1)o(v_2)=2^{m+1},$ we have $|\co^*|=2^{m+1}2^{\frac{mn}{2}-1}=(2^m)^{\frac{n+2}{2}}$ using Proposition \ref{p} again. From this, it follows that $\co^*$ is a self-dual code of length $n+2$ over $\z.$

Next we observe that $wt_E(v_{1})=2^{m-1}$ and $wt_E(v_{2})=2^m.$ Also for every $c_2 \in \co_2,$ we note that $wt_E(c_2) \not\equiv 0~(\text{mod }2^{m+1})$ and  $wt_E(c_2)\equiv c_2\cdot c_2 \equiv 0~(\text{mod }2^m)$ using Theorem \ref{theosorel}(a), which implies that $wt_E(c_2) \equiv 2^m~(\text{mod }2^{m+1}).$ This gives $wt_E(v_{2},c_2)=wt_E(v_{2})+wt_E(c_2)\equiv 2^m+2^m\equiv 0 ~(\text{mod }2^{m+1})$ for every $c_2 \in \co_2.$ On the other hand, from Lemma \ref{2.2}, we have $wt_E(c_1) \equiv 2^{m-2}n ~(\text{mod }2^{m+1})$ for every $c_1 \in \co_1.$ This implies that for every $c_1 \in \co_1,$ we have $wt_E(v_{1},c_1)=wt_E(v_{1})+wt_E(c_1)\equiv 2^{m-1}+2^{m-2}n ~(\text{mod }2^{m+1}),$ which gives  \vspace{-2mm}$$wt_E(v_{1},c_1)\equiv \left\{\begin{array}{ll} 0~(\text{mod }2^{m+1}) & \text{if } n \equiv 6~(\text{mod }8);\\2^m~(\text{mod }2^{m+1}) & \text{if } n \equiv 2~(\text{mod }8). \end{array} \right.\vspace{-2mm}$$
This proves that the code $\co^*$ is a Type II code when $n\equiv 6~(\text{mod }8)$ and $\co^*$ is a Type I code when $n\equiv 2~(\text{mod }8).$

In order to compute the complete weight enumerator of $\co^*,$ by Proposition \ref{p}, we note that any element in $\co^*$ is of the form $c^*=(iv_1+jv_2,c_{ij}),$ where $c_{ij} \in \co_{\eta(i,j)},$ $1\leq i\leq {o(v_1)}$ and $1\leq j\leq {o(v_2)}.$ This gives
\vspace{-2mm}\begin{eqnarray*}
c^*=\left\{\begin{array}{ll}((i+2j)2^{\f-1},i2^{\f-1},c_{ij}) & \text{if }m\text{ is even;}\vspace{2mm}\\((i+j)2^{\fo},
      j2^{\fo},c_{ij})&
      \text{if }m \text{ is odd,}\end{array}\right.
\vspace{-2mm}\end{eqnarray*} where $c_{ij}  \in \co_{\eta(i,j)},$ $1\leq i\leq {o(v_1)}$ and $1\leq j\leq {o(v_2)}.$  From this,  the desired result follows immediately.
\vspace{-2mm}\item[(b)] It is clear that vectors $v_{1}$ and $v_{2}$ satisfy the properties $(\mathbf{P_1}),$ $(\mathbf{P_2})$ and $(\mathbf{P_3}).$ Thus in view of Proposition \ref{p}, we see that the code $\co^*$ is a self-orthogonal code of length $n+6$ over $\z.$ Also since  $o(v_1)o(v_2)=2^{m+1},$ we have $|\co^*|=2^{m+1}2^{\frac{mn}{2}-1}=(2^m)^{\frac{n+2}{2}}$ using Proposition \ref{p} again. Now we need to show that the code $\co'=\left<\co^*\cup \{w_1,w_2,w_3,w_4\}\right>$ is a self-dual code of length $n+6$ over $\z.$ To prove this, we first observe that $w_p \cdot w_q=w_p \cdot c^*=0$ for $1\leq p,q\leq 4$ and for each $c^* \in \co^*.$ This implies that the code $\co'$ is a self-orthogonal code of length $n+6$ over $\z.$ Also it is easy to observe that $|\co'|=o(w_1)o(w_2)o(w_3)o(w_4)|\co^*|,$ where $o(w_p)$ is the additive order of $w_p$ for each $p.$ Now as $o(w_1)o(w_2)o(w_3)o(w_4)=2^{2m},$ we obtain $|\co'|=2^{2m}|\co^*|=(2^m)^{\frac{n+6}{2}}.$ This implies that $\co'$ is a self-dual code of length $n+6$ over $\z.$ Further working in a similar way as in part (a) and using the fact that  $wt_E(w_1)\equiv wt_E(w_2)\equiv wt_E(w_3)\equiv wt_E(w_4)\equiv 0~(\text{mod }2^{m+1}),$ we see that the code $\co'$ is a Type II code when $n\equiv 2~(\text{mod }8)$ and $\co'$ is a Type I code when $n\equiv 6~(\text{mod }8).$

Next to compute the complete weight enumerator of $\co',$ by Proposition \ref{p}, we note that any element in $\co'$ is of the form $c'=(iv_1+jv_2,c_{ij})+k_1w_1+k_2w_2+k_3w_3+k_4w_4,$ where $c_{ij} \in \co_{\eta(i,j)},$ $1\leq i\leq {o(v_1)},$ $1\leq j\leq {o(v_2)}$ and $1\leq k_p\leq o(w_p)$ for $1 \leq p \leq 4.$ Then working in a similar way as in part (a), we obtain the desired result. \end{itemize}
\vspace{-2mm}{\bf II.} Let $n\equiv 0~(\text{mod }4).$ Here also, we will first construct a self-orthogonal code $\co^*$ by choosing a suitable positive integer $k$ and vectors $v_1,$ $v_2$ satisfying the properties $(\mathbf{P_1}),$ $(\mathbf{P_2})$ and $(\mathbf{P_3}).$ Now as $s \in \co_1$ and $t \in \co_2,$ by Theorem \ref{theosorel}(b), we see that $s\cdot s\equiv 0~(\text{mod }2^m),$ $s\cdot t\equiv 2^{m-1}~(\text{mod }2^m)$ and $t\cdot t\equiv  0~(\text{mod }2^m).$ Therefore the property $(\mathbf{P_2})$ becomes \vspace{-2mm}\begin{equation*}
v_{1}\cdot v_{1}\equiv  0~(\text{mod }2^m),~~~
v_{1}\cdot v_{2}\equiv  2^{m-1}~(\text{mod }2^m),~~~
v_{2}\cdot v_{2}\equiv 0~(\text{mod }2^m).\vspace{-2mm}
\end{equation*}
Now working similarly as in part I, one can obtain the desired results.
\end{proof}
Next we proceed to determine Jacobi forms from complete  weight enumerators of Type II codes over $\z$ that are constructed in the above theorem. For this purpose, we need the following theta series defined by Choie and Kim \cite{choie}.
\begin{defin}\cite{choie}\label{defthet}
For each $\mu\in \mathbb{Z}_{2^m},$  the theta series $\theta_{2^{m-1},\mu}$ is a function from $\mathfrak{H} \times \mathbb{C}$ into $\mathbb{C}$ defined as
\vspace{-2mm}$$\theta_{2^{m-1},\mu}(\tau,z)=\sum_{\substack{r\in\mathbb{Z}\\ r\equiv \mu~(\text{mod }2^m)}}q^{\frac{r^2}{2^{m+1}}}\xi^{r},\text{~ where }q=e^{2\pi i\tau}\text{ and }\xi=e^{2\pi i z}.\vspace{-2mm}$$
\end{defin}
\noindent In the following theorem, we obtain some Jacobi forms on the modular group $\Gamma(1).$
\begin{theo}\label{theosnej}
Let $\co$ be a Type I code of even length $n$ over $\z$ and $\co_0=\{c \in \co: wt_E(c) \equiv 0~(\text{mod }2^{m+1})\}.$  Then the following holds:
\begin{itemize}
  \vspace{-2mm}\item[\textup{(a)}] When $n\equiv 6~(\text{mod }8),$ let $\co^*$ be the Type II code of length $n+2$ over $\z$ as defined in  Theorem \ref{theosneven} I\textup{(a)}. Let $cwe_{\co^*}(X_{\mu}:\mu\in\z)$  be as obtained in Theorem \ref{theosneven} I\textup{(a)}. Then
      $cwe_{\co^*}(\theta_{2^{m-1},\mu}(\tau,z):\mu\in \z)$ is a Jacobi form of weight $\frac{n+2}{2}$ and index $(n+2)2^{m-1}$ on $\Gamma(1).$
 \vspace{-2mm}\item[\textup{(b)}] When $n\equiv 2~(\text{mod }8),$ let $\co'$ be the Type II code of length $n+6$ over $\z$ as defined in Theorem \ref{theosneven} I\textup{(b)}.  Let $cwe_{\co'}(X_{\mu}:\mu\in\z)$  be as obtained in Theorem \ref{theosneven} I\textup{(b)}. Then
      $cwe_{\co'}(\theta_{2^{m-1},\mu}(\tau,z):\mu\in \z)$ is a Jacobi form of weight $\frac{n+6}{2}$ and index $(n+6)2^{m-1}$ on $\Gamma(1).$
\vspace{-2mm}\item[\textup{(c)}] When $n\equiv 4~(\text{mod }8),$ let $\co'$ be the Type II code of length $n+4$ over $\z$ as defined in Theorem \ref{theosneven} II\textup{(a)}.  Let $cwe_{\co'}(X_{\mu}:\mu\in\z)$  be as obtained in Theorem \ref{theosneven} II\textup{(a)}. Then
      $cwe_{\co'}(\theta_{2^{m-1},\mu}(\tau,z):\mu\in \z)$ is a Jacobi form of weight $\frac{n+4}{2}$ and index $(n+4)2^{m-1}$ on $\Gamma(1).$
   \vspace{-2mm}\item[\textup{(d)}] When $n\equiv 0~(\text{mod }8),$ let $\co'$ be the Type II code of length $n+8$ over $\z$ as defined in Theorem \ref{theosneven} II\textup{(b)}.  Let $cwe_{\co'}(X_{\mu}:\mu\in\z)$ be as obtained in Theorem \ref{theosneven} II\textup{(b)}. Then
      $cwe_{\co'}(\theta_{2^{m-1},\mu}(\tau,z):\mu\in \z)$ is a Jacobi form of weight $\frac{n+8}{2}$ and index $(n+8)2^{m-1}$ on $\Gamma(1).$
\end{itemize}
\end{theo}
\begin{proof} By Proposition 4.2 and Theorem 4.4 of Choie and Kim \cite{choie}, we see that $cwe_{\co'}(\theta_{2^{m-1},\mu}(\tau,z):\mu\in \z)$ is a Jacobi form of weight $\frac{\ell}{2}$ and index $\ell2^{m-1}$ on $\Gamma(1),$ where $\ell$ is the length of the code $\co'.$
\end{proof}
In the following theorem, we construct Type I and Type II  codes (of higher lengths) over $\z$ from Type I codes of odd length $n$ over $\z.$ Here by Lemma \ref{a}, $m$ must be an even integer.
\begin{theo}\label{theosnodd}
Let $\co$ be a Type I code of odd length $n$ over $\z$ $(m$ is an even integer$)$ and $\co_0=\{c \in \co: wt_E(c) \equiv 0~(\text{mod }2^{m+1})\}.$ Here we have $\eta(i,j)=[i+2j]_4$ for $0\leq i,j\leq 2^m-1,$ and the following hold:
\begin{description}
  \vspace{-2mm}\item[I.] Let $n\equiv 3~(\text{mod }4).$
  \begin{itemize}
  \vspace{-2mm}\item[\textup{(a)}] Let $v_1,v_2 \in \z^5$ and $w_1,w_2,w_3 \in \z^{n+5}$ be chosen as \vspace{-2mm}\begin{eqnarray*}
 \begin{array}{ll}
v_1=(2^{\f-1},2^{\f-1},2^{\f-1},2^{\f-1},2^{\f-1}), &v_2=(2^{\f},0,0,0,0), \\
 w_1=(2^{\f},2^{\f},0,0,0,0,\cdots,0),&
w_2=(0,2^{\f},2^{\f},0,0,0,\cdots,0),\\
w_3=(0,0,2^{\f},2^{\f},0,0,\cdots,0).\end{array}
\end{eqnarray*}
 Then the code $\co'=\left< \co^* \cup \{w_1,w_2,w_3\}\right>$ is a self-dual code of length $n+5.$ Furthermore, the code $\co'$ is a Type II code when $n\equiv 3~(\text{mod }8)$ and $\co'$ is a Type I code when $n\equiv 7~(\text{mod }8).$ The complete weight enumerator $cwe_{\co'}(X_{\mu}:\mu\in\z)$ of $\co'$ is given by
\vspace{-2mm}\begin{eqnarray*}cwe_{\co'}(X_{\mu}:\mu\in\z)=\sum X_{(i+2j+2k_1)2^{\f-1}}X_{(i+2k_1+2k_2)2^{\f-1}}
     X_{(i+2k_2+2k_3)2^{\f-1}}X_{(i+2k_3)2^{\f-1}}\\ X_{i2^{\f-1}}cwe_{\co_{\eta(i,j)}}(X_{\mu}:\mu\in\z),
     \vspace{-2mm} \end{eqnarray*}
       where the summation $\sum$ runs over all integral 5-tuples $(i,j,k_1,k_2,k_3)$ satisfying $1\leq j,k_1,k_2,k_3\leq {2^{\f}}$ and $1\leq i\leq {2^{\f+1}}.$
 \item[\textup{(b)}] Let $v_1,v_2 \in \z^9$  and $w_p\in \z^{n+9}$ $(1\leq p\leq 7)$ be chosen as
  \begin{eqnarray*}
  \begin{array}{ll}
  v_1=(2^{\f-1},2^{\f-1},\cdots,2^{\f-1}),& v_2=(2^{\f},0,\cdots,0),\\
   w_1=(0,2^{\f},2^{\f},0\cdots,0),&  w_2=(0,0,2^{\f},2^{\f},0\cdots,0),\\
    w_3=(0,0,0,2^{\f},2^{\f},0\cdots,0), &  w_4=(0,0,0,0,2^{\f},2^{\f},0\cdots,0),\\
    w_5=(0,0,0,0,0,2^{\f},2^{\f},0\cdots,0), & w_6=(0,0,0,0,0,0,2^{\f},2^{\f},0\cdots,0),\\
    w_7=(0,0,0,0,0,0,0,2^{\f},2^{\f},0\cdots,0).\end{array}\end{eqnarray*}
Then $\co'=\langle \co^*\cup\{w_1,w_2,w_3,w_4,w_5,w_6,w_7\}\rangle$ is a self-dual code of length $n+9.$ Furthermore, $\co'$ is a Type II code when $n\equiv 7~(\text{mod }8)$  and $\co'$ is a Type I code when $n\equiv 3~(\text{mod }8).$ The complete weight enumerator $cwe_{\co'}(X_{\mu}:\mu\in\z)$ of $\co'$ is given by
      \vspace{-2mm}\begin{eqnarray*}cwe_{\co'}(X_{\mu}:\mu\in\z)=\sum X_{(i+2j)2^{\f-1}}X_{(i+2k_1)2^{\f-1}}
     X_{(i+2k_1+2k_2)2^{\f-1}}X_{(i+2k_2+2k_3)2^{\f-1}}\\
     X_{(i+2k_3+2k_4)2^{\f-1}}X_{(i+2k_7)2^{\f-1}}
     X_{(i+2k_4+2k_5)2^{\f-1}}X_{(i+2k_5+2k_6)2^{\f-1}}\\
     X_{(i+2k_6+2k_7)2^{\f-1}} cwe_{\co_{\eta(i,j)}}(X_{\mu}:\mu\in\z),
     \vspace{-2mm} \end{eqnarray*}
      where the summation $\sum$ runs over all integral 9-tuples $(i,j,k_1,k_2,k_3,k_4,k_5,k_6,k_7)$ satisfying $1\leq j,k_1,k_2,k_3,k_4,k_5,k_6,k_7\leq {2^{\f}}$ and $1\leq i\leq {2^{\f+1}}.$
\end{itemize}
\vspace{-2mm}\item[II.] Let $n\equiv 1~(\text{mod }4).$
  \begin{itemize}
\vspace{-2mm}\item[\textup{(a)}] Let $v_1,v_2 \in \z^3$ and $w_1\in \z^{n+3}$ be chosen as \begin{eqnarray*}
    \begin{array}{lll}
    v_1=(2^{\f-1},2^{\f-1},2^{\f-1}),&
v_2=(2^{\f},0,0),&w_1=(2^{\f},2^{\f},0,0,\cdots,0).\end{array}
 \end{eqnarray*} Then $\co'=\left< \co^* \cup w_1\right>$ is a self-dual code of length $n+3$ over $\z.$ Furthermore, $\co'$ is a Type II code when $n\equiv 5~(\text{mod }8)$ and $\co'$ is a Type I code when $n\equiv 1~(\text{mod }8).$ The complete weight enumerator $cwe_{\co'}(X_{\mu}:\mu\in\z)$ of $\co'$ is given by \small
\vspace{-2mm}\begin{eqnarray*} cwe_{\co'}(X_{\mu}:\mu\in\z)=\sum X_{(i+2j+2k_1)2^{\f-1}}X_{(i+2k_1)2^{\f-1}}
      X_{i2^{\f-1}}cwe_{\co_{\eta(i,j)}}(X_{\mu}:\mu\in\z),\end{eqnarray*}\normalsize
  where the summation $\sum$ runs over all integral 3-tuples $(i,j,k_1)$ satisfying $1\leq j,k_1\leq 2^{\f}$ and $1\leq i\leq 2^{\f+1}.$
\item[\textup{(b)}] Let $v_1,v_2 \in \z^7$ and $w_1,w_2,w_3,w_4,w_5 \in \z^{n+7}$ be chosen as
  \begin{eqnarray*}\begin{array}{ll}
 v_1=(2^{\f-1},2^{\f-1},2^{\f-1},2^{\f-1},
  2^{\f-1},2^{\f-1},2^{\f-1}),& v_2=(2^{\f},0,0,0,0,0,0),\\w_1=(2^{\f},2^{\f},0,\cdots,0),&
 w_2=(0,2^{\f},2^{\f},0,\cdots,0),\\
 w_3=(0,0,2^{\f},2^{\f},0,\cdots,0),&
 w_4=(0,0,0,2^{\f}, 2^{\f},0,\cdots,0),\\
 w_5=(0,0,0,0,2^{\f},2^{\f},0,\cdots,0).\end{array}
 \end{eqnarray*}
 Then $\co'=\left<\co^* \cup \{w_1,w_2,w_3,w_4,w_5\} \right>$ is a self-dual code of length $n+7$ over $\z.$ Furthermore, $\co'$ is a Type II code of length $n+7$ over $\z$ when $n\equiv 1~(\text{mod }8)$ and $\co'$ is a Type I code when $n\equiv 5~(\text{mod }8).$ The complete weight enumerator $cwe_{\co'}(X_{\mu}:\mu\in\z)$ of $\co'$ is given by
  \vspace{-2mm} \small{  \begin{eqnarray*}
      cwe_{\co'}(X_{\mu}:\mu\in\z)=\sum X_{(i+2j+2k_1)2^{\f-1}}X_{(i+2k_1+2k_2)2^{\f-1}}
      X_{(i+2k_2+2k_3)2^{\f-1}}X_{(i+2k_3+2k_4)2^{\f-1}}
      \\X_{(i+2k_4+2k_5)2^{\f-1}}X_{(i+2k_5)2^{\f-1}} X_{i2^{\f-1}} cwe_{\co_{\eta(i,j)}}(X_{\mu}:\mu\in\z),\end{eqnarray*}}\normalsize
  where the summation $\sum$ runs over all integral 7-tuples $(i,j,k_1,k_2,k_3,k_4,k_5)$ satisfying $1\leq j,k_1,k_2,k_3,k_4,k_5\leq {2^{\f}}$ and $1\leq i\leq {2^{\f+1}}.$
\end{itemize}
\end{description}
\end{theo}
\begin{proof} As $n$ is odd, by Lemma \ref{lemsglue}(b), we see that the glue group $\co_0^{\perp}/\co_0$ is a cyclic group of order 4. Thus by \eqref{n}, we have $\eta(i,j)=\left[i+2j\right]_{4}$ for each $i$ and $j.$  Then working in a similar way as in Theorem \ref{theosneven}, the result follows.
\end{proof}
In the following theorem, we obtain Jacobi forms from complete weight enumerators of Type II codes over $\z,$ which are constructed in the above theorem.
\begin{theo}\label{theosnoj}
Let $\co$ be a Type I code of odd length $n$ over $\z$ and $\co_0=\{c \in \co: wt_E(c) \equiv 0~(\text{mod }2^{m+1})\}.$  Then the following holds:
\begin{itemize}
  \vspace{-2mm}\item[\textup{(a)}] When $n\equiv 3~(\text{mod }8),$ let $\co'$ be the Type II code of length $n+5$ over $\z$ as defined in Theorem \ref{theosnodd} I\textup{(a)}. Let $cwe_{\co'}(X_{\mu}:\mu\in\z)$ be as obtained in Theorem \ref{theosnodd} I\textup{(a)}. Then
      $cwe_{\co'}(\theta_{2^{m-1},\mu}(\tau,z):\mu\in \z)$ is a Jacobi form of weight $\frac{n+5}{2}$ and index $(n+5)2^{m-1}$ on $\Gamma(1).$
 \vspace{-2mm}\item[\textup{(b)}] When $n\equiv 7~(\text{mod }8),$ let $\co'$ be the Type II code of length $n+9$ over $\z$ as defined in Theorem \ref{theosnodd} I\textup{(b)}. Let $cwe_{\co'}(X_{\mu}:\mu\in\z)$ be as obtained in Theorem \ref{theosnodd} I\textup{(b)}. Then
      $cwe_{\co'}(\theta_{2^{m-1},\mu}(\tau,z):\mu\in \z)$ is a Jacobi form of weight $\frac{n+9}{2}$ and index $(n+9)2^{m-1}$ on $\Gamma(1).$
\vspace{-2mm}\item[\textup{(c)}] When $n\equiv 5~(\text{mod }8),$ let $\co'$ be the Type II code of length $n+3$ over $\z$ as defined in Theorem \ref{theosnodd} II\textup{(a)}. Let $cwe_{\co'}(X_{\mu}:\mu\in\z)$ be as obtained in Theorem \ref{theosnodd} II\textup{(a)}. Then
      $cwe_{\co'}(\theta_{2^{m-1},\mu}(\tau,z):\mu\in \z)$ is a Jacobi form of weight $\frac{n+3}{2}$ and index $(n+3)2^{m-1}$ on $\Gamma(1).$
   \vspace{-2mm}\item[\textup{(d)}] When $n\equiv 1~(\text{mod }8),$ let $\co'$ be the Type II code of length $n+7$ over $\z$ as defined in Theorem \ref{theosnodd} II\textup{(b)}. Let $cwe_{\co'}(X_{\mu}:\mu\in\z)$ be as obtained in Theorem \ref{theosnodd} II\textup{(b)}. Then
      $cwe_{\co'}(\theta_{2^{m-1},\mu}(\tau,z):\mu\in \z)$ is a Jacobi form of weight $\frac{n+7}{2}$ and index $(n+7)2^{m-1}$ on $\Gamma(1).$
\end{itemize}
\end{theo}
\begin{proof}
Its proof is similar to that of Theorem \ref{theosnej}.\end{proof}
\section{Construction of self-dual codes from generalized shadows of self-dual codes}\label{secgenshad}
In this section, we will construct self-dual codes  (of higher lengths) over $\z$ from the generalized shadow $S_g(s)$ of a self-dual code $\co$ of length $n$ over $\z$  for all $n,$ where $s \in \z^n\setminus \co.$  Here we will consider the following two cases separately: (i) $s\cdot s \equiv 0~(\text{mod }2^m)$ and (ii) $s \cdot s \equiv 2^{m-1}~(\text{mod }2^m).$
\\In the following theorem, we consider the case $s\cdot s\equiv 0~(\text{mod }2^m).$
\begin{theo}\label{theogs0}
Let $\co$ be a self-dual code of  length $n$ over $\z.$ Let $\mathcal{S}_g(s)$ be the generalized shadow of $\co$ with respect to a vector $s \in \z^n \setminus \co$  satisfying $s\cdot s\equiv 0~(\text{mod }2^m).$ Let $\eta(i,j)$ $(0 \leq i,j \leq 2^m-1)$ be as defined in \eqref{n}. Then we have the following:
\begin{description}
 \vspace{-2mm} \item[I.] Let $n\equiv 0~(\text{mod }4).$
   \begin{itemize}
  \vspace{-2mm}\item[\textup{(a)}] Let $v_1,v_2 \in \z^4$ and $w_1,w_2 \in \z^{n+4}$ be as chosen in Theorem \ref{theosneven} II(a).
Then $\co'=\left< \co^* \cup \{w_1,w_2\}\right>$ is a self-dual code of length $n+4$ over $\z.$ Moreover, the complete weight enumerator $cwe_{\co'}(X_\mu:\mu\in\z)$ of $\co'$ is given by
$$cwe_{\co'}(X_\mu:\mu\in\z)=\left\{\begin{array}{ll}\sum_1 X_{(i+2j+2k_1)2^{\f-1}}X_{(i+2k_1)2^{\f-1}}\\
      X_{(i+2k_2)2^{\f-1}}^2cwe_{\co_{\eta(i,j)}}(X_\mu:\mu\in\z)& \text{if }m \text{ is even;}\vspace{2mm}\\\sum_2
      X_{(i+j+k_1+2k_2)2^{\fo}}X_{(j+k_1)2^{\fo}}
      X_{(i+k_1)2^{\fo}}\\X_{k_12^{\fo}}cwe_{\co_{\eta(i,j)}}(X_\mu:\mu\in\z) & \text{if }m \text{ is odd,}\end{array}\right.$$
 where the summation $\sum_1$ runs over all integral 4-tuples $(i,j,k_1,k_2)$ satisfying $1\leq j,k_1,k_2\leq 2^{\f}$ and $1\leq i\leq 2^{\f+1},$ whereas the summation $\sum_2$ runs over all integral 4-tuples $(i,j,k_1,k_2)$ satisfying $1\leq i,j,k_1\leq {2^{\fe}}$ and $1\leq k_2\leq {2^{\fo}}.$
   \item[\textup{(b)}] Let $v_1,v_2 \in \z^8$ and $w_1,w_2,w_3,w_4,w_5,w_6 \in \z^{n+8}$ be as chosen in Theorem \ref{theosneven} II(b).
 Then $\co'=\left<\co^* \cup \{w_1,w_2,w_3,w_4,w_5,w_6\} \right>$ is a self-dual code of length $n+8$ over $\z.$ Moreover, the complete weight enumerator $cwe_{\co'}(X_\mu:\mu\in\z)$ of $\co'$ is given by \small
      \begin{eqnarray*}
      cwe_{\co'}(X_\mu:\mu\in\z)=\left\{\begin{array}{ll}\sum_1 X_{(i+2j+2k_1)2^{\f-1}}X_{(i+2k_2)2^{\f-1}}
      X_{(i+2(k_1+k_3))2^{\f-1}}\\X_{(i+2(k_2+k_4))2^{\f-1}}X_{(i+2(k_3+k_5))2^{\f-1}}X_{(i+2(k_4+k_6))2^{\f-1}}\\ X_{(i+2k_5)2^{\f-1}}X_{(i+2k_6)2^{\f-1}} cwe_{\co_{\eta(i,j)}}(X_\mu:\mu\in\z) & \text{if } m \text{ is even;}\vspace{2mm}\\\sum_2 X_{(i+j+k_1+2k_4)2^{\fo}}X_{(j+k_1+2k_5)2^{\fo}}
      X_{(i+k_1+k_2+2k_6)2^{\fo}}
      \\X_{(k_1+k_2)2^{\fo}}X_{(i+k_2+k_3)2^{\fo}}
      X_{(k_2+k_3)2^{\fo}}X_{(i+k_3)2^{\fo}}\\
      X_{k_32^{\fo}}cwe_{\co_{\eta(i,j)}}(X_\mu:\mu\in\z) & \text{if }m \text{ is odd,}\end{array}\right.\end{eqnarray*}\normalsize
 where the summation $\sum_1$ runs over all integral 8-tuples $(i,j,k_1,k_2,k_3,k_4,k_5,k_6)$ satisfying $1\leq j,k_1,k_2,k_3,k_4,k_5,k_6\leq {2^{\f}}$ and $1\leq i\leq {2^{\f+1}},$ whereas the summation $\sum_2$ runs over all integral 8-tuples $(i,j,k_1,k_2,k_3,k_4,k_5,k_6)$ satisfying $1\leq i,j,k_1,k_2,k_{3}\leq {2^{\fe}}$ and $1\leq k_4,k_{5},k_6\leq {2^{\fo}}.$
 \end{itemize}
\vspace{-2mm}\item[II.] Let $n\equiv 2~(\text{mod }4).$
\begin{itemize}
\vspace{-2mm}\item[\textup{(a)}] Let $v_1,v_2 \in \z^6$ and $w_1,w_2,w_3,w_4 \in \z^{n+6}$ be chosen as \begin{eqnarray*}
 v_1&=&\left\{\begin{array}{ll}(2^{\f-1},2^{\f-1},2^{\f-1},2^{\f-1},0,0) & \text{if } m \text{ is even;} \\(2^{\fo},0,2^{\fo},0,0,0) & \text{if }m \text{ is odd,} \end{array}\right.\\
 v_2&=&\left\{\begin{array}{ll}(2^{\f},0,0,0,0,0) & \text{if } m \text{ is even;} \\(2^{\fo},2^{\fo},0,0,0,0) & \text{if }m \text{ is odd,} \end{array}\right.\\
 w_1&=&\left\{\begin{array}{ll}(2^{\f},2^{\f},0,0,\cdots,0) & \text{if } m \text{ is even;} \\(2^{\fo},2^{\fo},2^{\fo},2^{\fo},0,\cdots,0) & \text{if }m \text{ is odd,} \end{array}\right.\\
 w_2&=&\left\{\begin{array}{ll}(0,2^{\f},2^{\f},0,0,\cdots,0) & \text{if } m \text{ is even;} \\(0,0,0,0,2^{\fo},2^{\fo},0,\cdots,0) & \text{if }m \text{ is odd,} \end{array}\right.\\
 w_3&=&\left\{\begin{array}{ll}(0,0,2^{\f},2^{\f},0,\cdots,0) & \text{if } m \text{ is even;} \\(2^{\fe},0,\cdots,0) & \text{if }m \text{ is odd,} \end{array}\right.\\
 w_4&=&\left\{\begin{array}{ll}(0,0,0,0,2^{\f},0,\cdots,0) & \text{if } m \text{ is even;} \\(0,2^{\fe},0,\cdots,0) & \text{if }m \text{ is odd.} \end{array}\right.
 \end{eqnarray*} Then $\co'=\left< \co^* \cup \{w_1,w_2,w_3,w_4\}\right>$ is a self-dual code of length $n+6$ over $\z.$ Moreover, the complete weight enumerator $cwe_{\co'}(X_{\mu}:\mu\in\z)$ of $\co'$ is given by
\vspace{-2mm}$$cwe_{\co'}(X_{\mu}:\mu\in\z)=\left\{\begin{array}{ll}\sum_1 X_0X_{(i+2j+2k_1)2^{\f-1}}X_{(i+2k_1+2k_2)2^{\f-1}}
    \\  X_{(i+2k_2+2k_3)2^{\f-1}}
    X_{(i+2k_3)2^{\f-1}}\\X_{k_42^{\f}}cwe_{\co_{\eta(i,j)}}(X_{\mu}:\mu\in\z)& \text{if }m \text{ is even;}\vspace{2mm}\\\sum_2
      X_{(i+j+k_1+2k_3)2^{\fo}}
      X_{(j+k_1+2k_4)2^{\fo}}\\
      X_{(i+k_1)2^{\fo}}
      X_{k_12^{\fo}}\\X_{k_22^{\fo}}^2cwe_{\co_{\eta(i,j)}}(X_{\mu}:\mu\in\z) & \text{if }m \text{ is odd,}\end{array}\right.\vspace{-2mm}$$
 where the summation $\sum_1$ runs over all integral 6-tuples $(i,j,k_1,k_2,k_3,k_{4})$ satisfying $1\leq j,k_p\leq 2^{\f}$ for $1\leq p\leq 4,$ and $1\leq i\leq 2^{\f+1},$ whereas the summation $\sum_2$ runs over all integral 6-tuples $(i,j,k_1,k_2,k_3,k_{4})$ satisfying $1\leq i,j,k_1,k_2\leq {2^{\fe}}$ and $1\leq k_3,k_4\leq {2^{\fo}}.$
 \item[\textup{(b)}] Let $v_1,v_2 \in \z^{10}$ and $w_p \in \z^{n+10}$ $(1\leq p\leq 8)$ be chosen as \vspace{-2mm}\begin{eqnarray*}
 v_1&=&\left\{\begin{array}{ll}(2^{\f-1},2^{\f-1},2^{\f-1},2^{\f-1},0,0,0,0,0,0) & \text{if } m \text{ is even;} \\(2^{\fo},0,2^{\fo},0,0,0,0,0,0,0) & \text{if }m \text{ is odd,} \end{array}\right.\\
 v_2&=&\left\{\begin{array}{ll}(2^{\f},0,0,0,0,0,0,0,0,0) & \text{if } m \text{ is even;} \\(2^{\fo},2^{\fo},0,0,0,0,0,0,0,0) & \text{if }m \text{ is odd,} \end{array}\right.\\
 w_1&=&\left\{\begin{array}{ll}(2^{\f},2^{\f},0,0,\cdots,0) & \text{if } m \text{ is even;} \\(2^{\fo},2^{\fo},2^{\fo},2^{\fo},0,\cdots,0) & \text{if }m \text{ is odd,} \end{array}\right.\\
 w_2&=&\left\{\begin{array}{ll}(0,2^{\f},2^{\f},0,0,\cdots,0) & \text{if } m \text{ is even;} \\(0,0,0,0,2^{\fo},2^{\fo},0,\cdots,0) & \text{if }m \text{ is odd,} \end{array}\right.\\
 w_3&=&\left\{\begin{array}{ll}(0,0,2^{\f},2^{\f},0,\cdots,0) & \text{if } m \text{ is even;} \\(0,0,0,0,0,0,2^{\fo},2^{\fo},0,\cdots,0) & \text{if }m \text{ is odd,} \end{array}\right.\\
 w_4&=&\left\{\begin{array}{ll}(0,0,0,0,2^{\f},0,\cdots,0) & \text{if } m \text{ is even;} \\(0,0,0,0,0,0,0,0,2^{\fo},2^{\fo},0,\cdots,0) & \text{if }m \text{ is odd,} \end{array}\right.\\
 w_5&=&\left\{\begin{array}{ll}(0,0,0,0,0,2^{\f},0,\cdots,0) & \text{if } m \text{ is even;} \\(2^{\fe},0,\cdots,0) & \text{if }m \text{ is odd,} \end{array}\right.\\
 w_6&=&\left\{\begin{array}{ll}(0,0,0,0,0,0,2^{\f},0,\cdots,0) & \text{if } m \text{ is even;} \\(0,2^{\fe},0,\cdots,0) & \text{if }m \text{ is odd,} \end{array}\right.\\
 w_7&=&\left\{\begin{array}{ll}(0,0,0,0,0,0,0,2^{\f},0,\cdots,0) & \text{if } m \text{ is even;} \\(0,0,2^{\fe},0,\cdots,0) & \text{if }m \text{ is odd,} \end{array}\right.\\
 w_8&=&\left\{\begin{array}{ll}(0,0,0,0,0,0,0,0,2^{\f},0,\cdots,0) & \text{if } m \text{ is even;} \\(0,0,0,2^{\fe},0,\cdots,0) & \text{if }m \text{ is odd.} \end{array}\right.
 \end{eqnarray*} Then $\co'=\left< \co^* \cup \{w_p:1\leq p\leq 8\}\right>$ is a self-dual code of length $n+10$ over $\z.$ Moreover, the complete weight enumerator $cwe_{\co'}(X_{\mu}:\mu\in\z)$ of $\co'$ is given by \small
\vspace{-2mm}$$cwe_{\co'}(X_{\mu}:\mu\in\z)=\left\{\begin{array}{ll}\sum_1 X_0 X_{(i+2j+2k_1)2^{\f-1}}
X_{(i+2k_1+2k_2)2^{\f-1}}\\
      X_{(i+2k_2+2k_3)2^{\f-1}}
      X_{(i+2k_3)2^{\f-1}}X_{k_42^{\f}}
      X_{k_52^{\f}}\\X_{k_62^{\f}}X_{k_72^{\f}} X_{k_82^{\f}} cwe_{\co_{\eta(i,j)}}(X_{\mu}:\mu\in\z)& \text{if }m \text{ is even;}\vspace{2mm}\allowdisplaybreaks\\\sum_2
      X_{(i+j+k_1+2k_5)2^{\fo}}
      X_{(j+k_1+2k_6)2^{\fo}}\\
      X_{(i+k_1+2k_7)2^{\fo}}
      X_{(k_1+2k_8)2^{\fo}}X_{k_22^{\fo}}^2\\ X_{k_32^{\fo}}^2X_{k_42^{\fo}}^2cwe_{\co_{\eta(i,j)}}(X_{\mu}:\mu \in\z) & \text{if }m \text{ is odd,}\end{array}\right.\vspace{-2mm}$$\normalsize
 where the summation $\sum_1$ runs over all integral 10-tuples $(i,j,k_1,k_2,\cdots,k_{8})$ satisfying $1\leq j,k_p\leq 2^{\f}$ for $1\leq p\leq 8,$ and $1\leq i\leq 2^{\f+1},$ whereas the summation $\sum_2$ runs over all integral 10-tuples $(i,j,k_1,k_2,\cdots,k_{8})$ satisfying $1\leq k_1,k_2,k_3,k_4\leq {2^{\fe}}$ and $1\leq k_5,k_6,k_7,k_8\leq {2^{\fo}}.$
 \end{itemize}
 \vspace{-2mm}\item[III.] Let $n\equiv 3~(\text{mod }4).$ Here the integer $m$ must be even.
\begin{itemize}
\vspace{-2mm}\item[\textup{(a)}] Let $v_1,v_2 \in \z^5$ and $w_1,w_2,w_3 \in \z^{n+5}$ be chosen as \vspace{-2mm}\begin{eqnarray*}\begin{array}{ll}
 v_1=(2^{\f-1},2^{\f-1},2^{\f-1},2^{\f-1},0),&
 v_2=(2^{\f},0,0,0,0),\\
 w_1=(2^{\f},2^{\f},0,\cdots,0),&
 w_2=(0,2^{\f},2^{\f},0,\cdots,0),\\
 w_3=(0,0,2^{\f},2^{\f},0,\cdots,0).\end{array}
 \end{eqnarray*} Then $\co'=\left< \co^* \cup \{w_1,w_2,w_3\}\right>$ is a self-dual code of length $n+5$ over $\z.$ Moreover, the complete weight enumerator $cwe_{\co'}(X_{\mu}:\mu\in\z)$ of $\co'$ is given by
\vspace{-2mm}\begin{eqnarray*}cwe_{\co'}(X_{\mu}:\mu\in\z)=\sum X_0 X_{(i+2j+2k_1)2^{\f-1}}X_{(i+2k_1+2k_2)2^{\f-1}}
      X_{(i+2k_2+2k_3)2^{\f-1}}\\X_{(i+2k_3)2^{\f-1}}  cwe_{\co_{\eta(i,j)}}(X_{\mu}:\mu\in\z),\vspace{-2mm}\end{eqnarray*}
 where the summation $\sum$ runs over all integral 5-tuples $(i,j,k_1,k_2,k_3)$ satisfying $1\leq j,k_p\leq 2^{\f}$ for $1\leq p\leq 3,$ and $1\leq i\leq 2^{\f+1}.$
\item[\textup{(b)}] Let $v_1,v_2 \in \z^{9}$ and $w_p \in \z^{n+9}$ $(1\leq p\leq 7)$ be chosen as \vspace{-2mm}\begin{eqnarray*}
     \begin{array}{ll}
 v_1=(2^{\f-1},2^{\f-1},2^{\f-1},2^{\f-1},0,0,0,0,0),&
 v_2=(2^{\f},0,0,0,0,0,0,0,0),\\
 w_1=(2^{\f},2^{\f},0,\cdots,0),&
 w_2=(0,2^{\f},2^{\f},0,\cdots,0),\\
 w_3=(0,0,2^{\f},2^{\f},0,\cdots,0),&
 w_4=(0,0,0,0,2^{\f},0,\cdots,0),\\
 w_5=(0,0,0,0,0,2^{\f},0,\cdots,0),&
 w_6=(0,0,0,0,0,0,2^{\f},0,\cdots,0),\\
 w_7=(0,0,0,0,0,0,0,2^{\f},0,\cdots,0).\end{array}
 \end{eqnarray*} Then $\co'=\left< \co^* \cup \{w_p:1\leq p\leq 7\}\right>$ is a self-dual code of length $n+9$ over $\z.$ Moreover, the complete weight enumerator $cwe_{\co'}(X_{\mu}:\mu\in\z)$ of $\co'$ is given by
\vspace{-2mm}\begin{eqnarray*}cwe_{\co'}(X_{\mu}:\mu\in\z)=\sum X_0 X_{(i+2j+2k_1)2^{\f-1}}X_{(i+2k_1+2k_2)2^{\f-1}}
      X_{(i+2k_2+2k_3)2^{\f-1}}\\X_{(i+2k_3)2^{\f-1}}
      X_{k_42^{\f}}X_{k_52^{\f}}X_{k_62^{\f}}X_{k_72^{\f}}
       cwe_{\co_{\eta(i,j)}}(X_{\mu}:\mu\in\z),\end{eqnarray*}
  where the summation $\sum$ runs over all integral 9-tuples $(i,j,k_1,k_2,\cdots,k_{7})$ satisfying $1\leq j,k_p\leq 2^{\f}$ for $1\leq p\leq 7,$ and $1\leq i\leq 2^{\f+1}.$
 \end{itemize}
 \vspace{-2mm}\item[IV.] Let $n\equiv 1~(\text{mod }4).$ Here $m$ must be an even integer.
\begin{itemize}
 \vspace{-2mm}\item[\textup{(a)}] Let $v_1,v_2 \in \z^7$ and $w_1,w_2,w_3,w_4,w_5 \in \z^{n+7}$ be chosen as \vspace{-2mm}\begin{eqnarray*}
    \begin{array}{ll}
 v_1=(2^{\f-1},2^{\f-1},2^{\f-1},2^{\f-1},0,0,0),&
 v_2=(2^{\f},0,0,0,0,0,0),\\
 w_1=(2^{\f},2^{\f},0,\cdots,0),&
 w_2=(0,2^{\f},2^{\f},0,\cdots,0),\\
 w_3=(0,0,2^{\f},2^{\f},0,\cdots,0),&
 w_4=(0,0,0,0,2^{\f},0,\cdots,0),\\
 w_5=(0,0,0,0,0,2^{\f},0,\cdots,0).\end{array}
 \end{eqnarray*} Then $\co'=\left< \co^* \cup \{w_1,w_2,w_3,w_4,w_5\}\right>$ is a self-dual code of length $n+7$ over $\z.$ Moreover, the complete weight enumerator $cwe_{\co'}(X_{\mu}:\mu\in\z)$ of $\co'$ is given by
\vspace{-2mm}\begin{eqnarray*}cwe_{\co'}(X_{\mu}:\mu\in\z)=\sum X_0X_{(i+2j+2k_1)2^{\f-1}}X_{(i+2k_1+2k_2)2^{\f-1}}
      X_{(i+2k_2+2k_3)2^{\f-1}}\\X_{(i+2k_3)2^{\f-1}}X_{k_42^{\f}}X_{k_52^{\f}} cwe_{\co_{\eta(i,j)}}(X_{\mu}:\mu\in\z),\end{eqnarray*}
  where the summation $\sum$ runs over all integral 7-tuples $(i,j,k_1,k_2,\cdots,k_{5})$ satisfying $1\leq j,k_p\leq 2^{\f}$ for $1\leq p\leq 5,$ and $1\leq i\leq 2^{\f+1}.$
 \item[\textup{(b)}] Let $v_1,v_2 \in \z^{11}$ and $w_p \in \z^{n+11}$ $(1\leq p\leq 9)$ be chosen as \vspace{-2mm}\begin{eqnarray*}
     \begin{array}{ll}
 v_1=(2^{\f-1},2^{\f-1},2^{\f-1},2^{\f-1},0,0,0,0,0,0,0),&
 v_2=(2^{\f},0,0,0,0,0,0,0,0,0,0),\\
 w_1=(2^{\f},2^{\f},0,\cdots,0),&
 w_2=(0,2^{\f},2^{\f},0,\cdots,0),\\
 w_3=(0,0,2^{\f},2^{\f},0,\cdots,0),&
 w_4=(0,0,0,0,2^{\f},0,\cdots,0),\\
 w_5=(0,0,0,0,0,2^{\f},0,\cdots,0),&
 w_6=(0,0,0,0,0,0,2^{\f},0,\cdots,0),\\
 w_7=(0,0,0,0,0,0,0,2^{\f},0,\cdots,0),&
 w_8=(0,0,0,0,0,0,0,0,2^{\f},0,\cdots,0),\\
 w_9=(0,0,0,0,0,0,0,0,0,2^{\f},0,\cdots,0).\end{array}
 \end{eqnarray*} Then $\co'=\left< \co^* \cup \{w_p:1\leq p\leq 9\}\right>$ is a self-dual code of length $n+11$ over $\z.$ Moreover, the complete weight enumerator $cwe_{\co'}(X_{\mu}:\mu\in\z)$ of $\co'$ is given by \small
\vspace{-2mm}\begin{eqnarray*}cwe_{\co'}(X_{\mu}:\mu\in\z)=\sum X_0X_{(i+2j+2k_1)2^{\f-1}}X_{(i+2k_1+2k_2)2^{\f-1}}
      X_{(i+2k_2+2k_3)2^{\f-1}}X_{(i+2k_3)2^{\f-1}}
      \\X_{k_42^{\f}}X_{k_52^{\f}}X_{k_62^{\f}}X_{k_72^{\f}} X_{k_82^{\f}}X_{k_92^{\f}}  cwe_{\co_{\eta(i,j)}}(X_{\mu}:\mu\in\z),\end{eqnarray*}\normalsize
  where the summation $\sum$ runs over all integral 11-tuples $(i,j,k_1,k_2,\cdots,k_{9})$ satisfying $1\leq j,k_p\leq 2^{\f}$ for $1\leq p\leq 9,$ and $1\leq i\leq 2^{\f+1}.$
 \end{itemize}
 \end{description}
\end{theo}
\begin{proof} Working in a similar way as in Theorem \ref{theosneven}, the result follows.
\end{proof}
In the following theorem, we consider the case $s\cdot s\equiv 2^{m-1}~(\text{mod }2^m).$
\begin{theo}\label{theogsn0}
Let $\co$ be a self-dual code of length $n$ over $\z.$ Let $\mathcal{S}_g(s)$ be the generalized shadow of $\co$ with respect to a vector  $s \in \z^n \setminus \co$ satisfying $s\cdot s\equiv 2^{m-1}~(\text{mod }2^m).$ Let $\eta(i,j)$ ($0 \leq i,j \leq 2^m-1$) be as defined in \eqref{n}. Then we have the following:
\begin{description}
  \vspace{-2mm}\item[I.] Let $n\equiv 2~(\text{mod }4).$
    \begin{itemize}
 \vspace{-2mm} \item[\textup{(a)}] Let $v_1,v_2 \in \z^2$  be as chosen in Theorem \ref{theosneven} I(a).
Then $\co^*$ is a self-dual code of length $n+2.$ Moreover, the complete weight enumerator $cwe_{\co^*}(X_\mu:\mu\in\z)$ of $\co^*$ is given by
\small      $$cwe_{\co^*}(X_\mu:\mu\in\z)=\left\{\begin{array}{ll}\sum\limits_{i=1}^{2^{\f+1}}\sum\limits_{j=1}^{2^{\f}}
      X_{(i+2j)2^{\f-1}}X_{i2^{\f-1}}cwe_{\co_{\eta(i,j)}}(X_\mu:\mu\in\z)& \text{if }m\text{ is even;}\vspace{2mm}\\\sum\limits_{i=1}^{2^{\fe}}
      \sum\limits_{j=1}^{2^{\fe}}X_{(i+j)2^{\fo}}
      X_{j2^{\fo}}cwe_{\co_{\eta(i,j)}}(X_\mu:\mu\in\z)&
      \text{if }m \text{ is odd.}\end{array}\right.$$\normalsize
 \vspace{-2mm}\item[\textup{(b)}] Let $v_1,v_2 \in \z^6$ and $w_1,w_2,w_3,w_4 \in \z^{n+6}$ be as chosen in Theorem \ref{theosneven} I(b).
 Then the code $\co'=\left< \co^* \cup \{w_1,w_2,w_3,w_4\}\right>$ is a self-dual code of length $n+6.$ Moreover, the complete weight enumerator $cwe_{\co'}(X_\mu:\mu\in\z)$ of $\co'$ is given by \small
$$cwe_{\co'}(X_\mu:\mu\in\z)=\left\{\begin{array}{ll}\sum_1 X_{(i+2j+2k_1)2^{\f-1}}X_{(i+2k_2)2^{\f-1}}
      X_{(i+2(k_1+k_3))2^{\f-1}}\\X_{(i+2(k_2+k_4))2^{\f-1}} X_{(i+2k_3)2^{\f-1}}X_{(i+2k_4)2^{\f-1}}\\cwe_{\co_{\eta(i,j)}}(X_\mu:\mu\in\z) & \text{if } m \text{ is even;}\vspace{2mm}\\
      \sum_2 X_{(i+j+k_1+2k_2)2^{\fo}}X_{(j+k_1+2k_4)2^{\fo}}
      X_{(i+k_1+k_3)2^{\fo}}\\X_{(k_1+k_3)2^{\fo}} X_{(i+k_3)2^{\fo}}X_{k_32^{\fo}}cwe_{\co_{\eta(i,j)}}(X_\mu:\mu\in\z) &\text{if }m\text{ is odd,}\end{array}\right.$$ \normalsize  where the summation $\sum_1$ runs over all integral 6-tuples $(i,j,k_1,k_2,k_3,k_4)$ satisfying $1\leq j,k_1,k_2,\\k_3,k_4\leq {2^{\f}}$ and $1\leq i\leq {2^{\f+1}},$ whereas the summation $\sum_2$ runs over all integral 6-tuples $(i,j,k_1,k_2,k_3,k_4)$ satisfying $1\leq i,j,k_1,k_{3}\leq {2^{\fe}}$ and $1\leq k_2,k_{4}\leq {2^{\fo}}.$
      \end{itemize}
\vspace{-2mm}\item[II.] Let $n\equiv 0~(\text{mod }4).$
    \begin{itemize}
 \vspace{-2mm} \item[\textup{(a)}] Let $v_1,v_2 \in \z^4$  and $w_1,w_2\in \z^{n+4}$ be chosen as
  \vspace{-2mm}\begin{eqnarray*}
  v_1&=&\left\{\begin{array}{ll}(2^{\f-1},2^{\f-1},0,0) & \text{if } m \text{ is even;}\\
  (2^{\fo},0,0,0) & \text{if } m \text{ is odd,}\end{array} \right.\allowdisplaybreaks\\
  v_2&=&\left\{\begin{array}{ll}(2^{\f},0,0,0) & \text{if } m \text{ is even;}\\
  (2^{\fo},2^{\fo},0,0) & \text{if } m \text{ is odd,}\end{array} \right.\allowdisplaybreaks\\
   w_1&=& \left\{\begin{array}{ll}(2^{\f},2^{\f},0,\cdots,0) & \text{if } m \text{ is even;}\\
  (0,0,2^{\fo},2^{\fo},0,\cdots,0) & \text{if } m \text{ is odd,}\end{array} \right.\allowdisplaybreaks\\
 w_2&=&\left\{\begin{array}{ll}(0,0,2^{\f},0,\cdots,0) & \text{if } m \text{ is even;}\\
  (2^{\fe},0,\cdots,0) & \text{if } m \text{ is odd.}\end{array} \right.
  \end{eqnarray*}
Then $\co'=\langle \co^*\cup\{w_1,w_2\}\rangle$ is a self-dual code of length $n+4.$ Moreover, the complete weight enumerator $cwe_{\co'}(X_{\mu}:\mu\in\z)$ of $\co'$ is given by
\vspace{-2mm}$$cwe_{\co'}(X_{\mu}:\mu\in\z)=\left\{\begin{array}{ll}\sum_1 X_0X_{(i+2j+2k_1)2^{\f-1}}X_{(i+2k_1)2^{\f-1}}
      X_{k_22^{\f}}\\cwe_{\co_{\eta(i,j)}}(X_{\mu}:\mu\in\z)& \text{if }m \text{ is even;}\vspace{4mm}\\\sum_2
      X_{(i+j+2k_2)2^{\fo}}X_{j2^{\fo}}
      X_{k_12^{\fo}}^2\\cwe_{\co_{\eta(i,j)}}(X_{\mu}:\mu\in\z) & \text{if }m \text{ is odd,}\end{array}\right.\vspace{-2mm}$$
 where the summation $\sum_1$ runs over all integral 4-tuples $(i,j,k_1,k_2)$ satisfying $1\leq j,k_1,k_2\leq 2^{\f}$ and $1\leq i\leq 2^{\f+1},$ whereas the summation $\sum_2$ runs over all integral 4-tuples $(i,j,k_1,k_2)$ satisfying $1\leq i,j,k_1\leq {2^{\fe}}$ and $1\leq k_2\leq {2^{\fo}}.$
 \item[\textup{(b)}] Let $v_1,v_2 \in \z^8$ and $w_p \in \z^{n+8}$ $(1\leq p\leq 6)$ be chosen as \vspace{-2mm}\begin{eqnarray*}
 v_1&=&\left\{\begin{array}{ll}(2^{\f-1},2^{\f-1},0,0,0,0,0,0) & \text{if } m \text{ is even;} \\(2^{\fo},0,0,0,0,0,0) & \text{if }m \text{ is odd,} \end{array}\right.\\
 v_2&=&\left\{\begin{array}{ll}(2^{\f},0,0,0,0,0,0,0) & \text{if } m \text{ is even;} \\(2^{\fo},2^{\fo},0,0,0,0,0,0) & \text{if }m \text{ is odd,} \end{array}\right.\\
 w_1&=&\left\{\begin{array}{ll}(2^{\f},2^{\f},0,\cdots,0) & \text{if } m \text{ is even;} \\(0,0,2^{\fo},2^{\fo},0,\cdots,0) & \text{if }m \text{ is odd,} \end{array}\right.\\
 w_2&=&\left\{\begin{array}{ll}(0,0,2^{\f},0,0,0,0,\cdots,0) & \text{if } m \text{ is even;} \\(0,0,0,0,2^{\fo},2^{\fo},0,\cdots,0) & \text{if }m \text{ is odd,} \end{array}\right.\\
  w_3&=&\left\{\begin{array}{ll}(0,0,0,2^{\f},0,0,0,\cdots,0) & \text{if } m \text{ is even;} \\(0,0,0,0,0,0,2^{\fo},2^{\fo}0, \cdots,0) & \text{if }m \text{ is odd,} \end{array}\right.\\
w_4&=&\left\{\begin{array}{ll}(0,0,0,0,2^{\f},0,0,\cdots,0) & \text{if } m \text{ is even;} \\(2^{\fe},0,\cdots,0) & \text{if }m \text{ is odd,} \end{array}\right.\\
w_5&=&\left\{\begin{array}{ll}(0,0,0,0,0,2^{\f},0,\cdots,0) & \text{if } m \text{ is even;} \\(0,2^{\fe},0,\cdots,0) & \text{if }m \text{ is odd,} \end{array}\right.\allowdisplaybreaks\\
w_6&=&\left\{\begin{array}{ll}(0,0,0,0,0,0,2^{\f},0,\cdots,0) & \text{if } m \text{ is even;} \\(0,0,2^{\fe},0,\cdots,0) & \text{if }m \text{ is odd.} \end{array}\right.\allowdisplaybreaks
\end{eqnarray*}
 Then the code $\co'=\left< \co^* \cup \{w_p:1\leq p\leq 6\}\right>$ is a self-dual code of length $n+8.$ Moreover, the complete weight enumerator $cwe_{\co'}(X_{\mu}:\mu\in\z)$ of $\co'$ is given by
\vspace{-2mm}$$cwe_{\co'}(X_{\mu}:\mu\in\z)=\left\{\begin{array}{ll}\sum_1 X_0X_{(i+2j+2k_1)2^{\f-1}}
X_{(i+2k_1)2^{\f-1}}\\
      X_{k_22^{\f}}X_{k_32^{\f}}
      X_{k_42^{\f}}X_{k_52^{\f}}X_{k_62^{\f}}\\cwe_{\co_{\eta(i,j)}}(X_{\mu}:\mu\in\z)& \text{if }m \text{ is even;}\vspace{2mm}\\\sum_2
      X_{(i+j+2k_4)2^{\fo}}X_{(j+2k_5)2^{\fo}}\\
      X_{(k_1+2k_6)2^{\fo}}X_{k_12^{\fo}}
      X_{k_22^{\fo}}^2\\X_{k_32^{\fo}}^2cwe_{\co_{\eta(i,j)}}(X_{\mu}:\mu\in\z) & \text{if }m \text{ is odd,}\end{array}\right.\vspace{-2mm}$$
 where the summation $\sum_1$ runs over all integral 8-tuples $(i,j,k_1,k_2,\cdots,k_6)$ satisfying
 $1\leq j,k_p\leq 2^{\f}$ for $1\leq p\leq 6$ and $1\leq i\leq 2^{\f+1},$ whereas the summation $\sum_2$ runs over all integral 8-tuples $(i,j,k_1,k_2,\cdots,k_6)$ satisfying $1\leq i,j,k_1,k_2,k_3\leq {2^{\fe}}$ and $1\leq k_4,k_5,k_6\leq {2^{\fo}}.$
 \end{itemize}
 \vspace{-2mm}\item[III.] Let $n\equiv 3~(\text{mod }4).$ Here $m$ must be an even integer.
\begin{itemize}
\vspace{-2mm}\item[\textup{(a)}] Let $v_1,v_2 \in \z^5$ and $w_1,w_2,w_3 \in \z^{n+5}$ be chosen as \vspace{-2mm}\begin{eqnarray*}\begin{array}{lll}
 v_1=(2^{\f-1},2^{\f-1},0,0,0),&
 v_2=(2^{\f},0,0,0,0),&
 w_1=(2^{\f},2^{\f},0,\cdots,0),\\
 w_2=(0,0,2^{\f},0,\cdots,0),&
 w_3=(0,0,0,2^{\f},0,\cdots,0).\end{array}
 \end{eqnarray*} Then $\co'=\left< \co^* \cup \{w_1,w_2,w_3\}\right>$ is a self-dual code of length $n+5$ over $\z.$ Moreover, the complete weight enumerator $cwe_{\co'}(X_{\mu}:\mu\in\z)$ of $\co'$ is given by
\vspace{-2mm}\begin{eqnarray*}cwe_{\co'}(X_{\mu}:\mu\in\z)=\sum X_0X_{(i+2j+2k_1)2^{\f-1}}X_{(i+2k_1)2^{\f-1}}
      X_{k_22^{\f-1}}X_{k_32^{\f-1}}\\cwe_{\co_{\eta(i,j)}}(X_{\mu}:\mu\in\z),\end{eqnarray*}\normalsize
 where the summation $\sum$ runs over all integral 5-tuples $(i,j,k_1,k_2,k_3)$ satisfying $1\leq j,k_p\leq 2^{\f}$ for $1\leq p\leq 3,$ and $1\leq i\leq 2^{\f+1}.$
 \item[\textup{(b)}] Let $v_1,v_2 \in \z^{9}$ and $w_p \in \z^{n+9}$ $(1\leq p\leq 7)$ be chosen as \vspace{-2mm}\begin{eqnarray*}
     \begin{array}{ll}
 v_1=(2^{\f-1},2^{\f-1},0,0,0,0,0,0,0),&
 v_2=(2^{\f},0,0,0,0,0,0,0,0),\\
 w_1=(2^{\f},2^{\f},0,\cdots,0),&
 w_2=(0,0,2^{\f},0,\cdots,0),\\
 w_3=(0,0,0,2^{\f},0,\cdots,0),&
 w_4=(0,0,0,0,2^{\f},0,\cdots,0),\\
 w_5=(0,0,0,0,0,2^{\f},0,\cdots,0),&
 w_6=(0,0,0,0,0,0,2^{\f},0,\cdots,0),\\
 w_7=(0,0,0,0,0,0,0,2^{\f},0,\cdots,0).\end{array}
 \end{eqnarray*} Then $\co'=\left< \co^* \cup \{w_p:1\leq p\leq 7\}\right>$ is a self-dual code of length $n+9$ over $\z.$ Moreover, the complete weight enumerator $cwe_{\co'}(X_{\mu}:\mu\in\z)$ of $\co'$ is given by
\vspace{-2mm}\begin{eqnarray*}cwe_{\co'}(X_{\mu}:\mu\in\z)=\sum X_0X_{(i+2j+2k_1)2^{\f-1}}X_{(i+2k_1)2^{\f-1}}
      X_{k_22^{\f}}X_{k_32^{\f}}X_{k_42^{\f}}\\X_{k_52^{\f}}X_{k_62^{\f}}X_{k_72^{\f}} cwe_{\co_{\eta(i,j)}}(X_{\mu}:\mu\in\z),\end{eqnarray*}
 where the summation $\sum$ runs over all integral 9-tuples $(i,j,k_1,k_2,\cdots,k_{7})$ satisfying $1\leq j,k_p\leq 2^{\f}$ for $1\leq p\leq 7,$ and $1\leq i\leq 2^{\f+1}.$
 \end{itemize}
 \vspace{-2mm}\item[IV.] Let $n\equiv 1~(\text{mod }4).$ Here $m$ must be an even integer.
\begin{itemize}
\vspace{-2mm}\item[\textup{(a)}] Let $v_1,v_2 \in \z^7$ and $w_1,w_2,w_3,w_4,w_5 \in \z^{n+7}$ be chosen as \vspace{-2mm}\begin{eqnarray*}
    \begin{array}{ll}
 v_1=(2^{\f-1},2^{\f-1},0,0,0,0,0), &
 v_2=(2^{\f},0,0,0,0,0,0),\\
 w_1=(2^{\f},2^{\f},0,\cdots,0), &
 w_2=(0,0,2^{\f},0,\cdots,0),\\
 w_3=(0,0,0,2^{\f},0,\cdots,0),&
 w_4=(0,0,0,0,2^{\f},0,\cdots,0),\\
 w_5=(0,0,0,0,0,2^{\f},0,\cdots,0).\end{array}
 \end{eqnarray*} Then $\co'=\left< \co^* \cup \{w_1,w_2,w_3,w_4,w_5\}\right>$ is a self-dual code of length $n+7$ over $\z.$ Moreover, the complete weight enumerator $cwe_{\co'}(X_{\mu}:\mu\in\z)$ of $\co'$ is given by
\vspace{-2mm}\begin{eqnarray*}cwe_{\co'}(X_{\mu}:\mu\in\z)=\sum X_0X_{(i+2j+2k_1)2^{\f-1}}
X_{(i+2k_1)2^{\f-1}}
      X_{k_22^{\f}}X_{k_32^{\f}}X_{k_42^{\f}}
      X_{k_52^{\f}}\\ cwe_{\co_{\eta(i,j)}}(X_{\mu}:\mu\in\z),\end{eqnarray*}
 where the summation $\sum$ runs over all integral 7-tuples $(i,j,k_1,k_2,\cdots,k_{5})$ satisfying
       $1\leq j,k_p\leq 2^{\f}$ for $1\leq p\leq 5,$ and $1\leq i\leq 2^{\f+1}.$
\item[\textup{(b)}] Let $v_1,v_2 \in \z^{3}$ and $w_1 \in \z^{n+3}$ be chosen as \begin{eqnarray*}
         v_1=(2^{\f-1},2^{\f-1},0),~~  v_2=(2^{\f},0,0),~~ w_1=(2^{\f},2^{\f},0,\cdots,0).
 \end{eqnarray*} Then $\co'=\left< \co^* \cup w_1\right>$ is a self-dual code of length $n+3$ over $\z.$ Moreover, the complete weight enumerator $cwe_{\co'}(X_{\mu}:\mu\in\z)$ of $\co'$ is given by
\vspace{-2mm}\begin{eqnarray*}cwe_{\co'}(X_{\mu}:\mu\in\z)=\sum X_0X_{(i+2j+2k_1)2^{\f-1}}X_{(i+2k_1)2^{\f-1}}
      cwe_{\co_{\eta(i,j)}}(X_{\mu}:\mu\in\z),\end{eqnarray*}
 where the summation $\sum$ runs over all integral 3-tuples $(i,j,k_1)$ satisfying $1\leq j,k_1\leq 2^{\f}$ and $1\leq i\leq 2^{\f+1}.$
 \end{itemize}
 \end{description}
\end{theo}
\begin{proof}
Working in a similar way as in Theorem \ref{theosneven}, the result follows.
\end{proof}
\begin{rem}
\noindent \begin{itemize}
\vspace{-2mm}\item[(i)] Theorems 1 and 2 of Brualdi and Pless \cite{bru} follow from Theorem \ref{theosneven} by taking $m=1.$
\vspace{-2mm}\item[(ii)] The main theorem of Tsai \cite{han} follows from Theorems \ref{theogs0} and \ref{theogsn0} by taking $m=1.$
\vspace{-2mm}\item[(iii)] One can prove Theorems \ref{theosneven}-\ref{theogsn0} for codes over the ring $\mathbb{Z}_{2^mk^2}$ ($k >1$ is an integer) by replacing vectors $v_1,v_2$ and $w_p$'s with $kv_1,kv_2$ and $kw_p$'s, respectively.
\end{itemize}
\end{rem}

\end{document}